\documentclass[reqno,english]{amsart}
\usepackage[utf8]{inputenc}

\usepackage{amsmath,amssymb,verbatim}
\usepackage{amsthm}
\usepackage{mathtools}
\usepackage{slashed}
\usepackage[shortlabels]{enumitem}
\usepackage{xfrac}
\usepackage{tabstackengine}
\usepackage{xcolor}
\usepackage[safe]{tipa}
\usepackage[bookmarks]{hyperref}
\usepackage[capitalize]{cleveref}

\setstacktabbedgap{6pt}
\setstackgap{L}{17pt}

\declareslashed{}{/}{0.05}{0}{A}

\makeatletter
\DeclareFontEncoding{LS2}{}{\@noaccents}
\makeatother
\DeclareFontSubstitution{LS2}{stix}{m}{n}

\DeclareSymbolFont{largesymbolsstix}{LS2}{stixex}{m}{n}

\DeclareMathDelimiter{\lbrbrak}{\mathopen}{largesymbolsstix}{"EE}{largesymbolsstix}{"14}
\DeclareMathDelimiter{\rbrbrak}{\mathclose}{largesymbolsstix}{"EF}{largesymbolsstix}{"15}

\def\tbr#1{[[#1]]}

\newcommand{\tleq}{\trianglelefteq}
\newcommand{\tgeq}{\trianglerighteq}

\newtheorem{thm}{Theorem}[section]
\newtheorem{prop}[thm]{Proposition}
\newtheorem{lem}[thm]{Lemma}
\newtheorem{cor}[thm]{Corollary}
\newtheorem{fact}[thm]{Fact}

\newtheorem{quest}[thm]{Question}
\theoremstyle{definition}
\newtheorem{defn}[thm]{Definition}
 
\theoremstyle{remark}

\newcommand{\heq}{\textnormal{heq}}

\newcommand{\eq}{\textnormal{eq}}
\newcommand{\Las}{\textnormal{L}}
\newcommand{\EM}{\textnormal{EM}}
\newcommand{\LEM}{\textnormal{LEM}}

\def\Ind{\setbox0=\hbox{$x$}\kern\wd0\hbox to 0pt{\hss$\mid$\hss}
  \lower.9\ht0\hbox to 0pt{\hss$\smile$\hss}\kern\wd0}
\def\Notind{\setbox0=\hbox{$x$}\kern\wd0\hbox to 0pt{\mathchardef
    \nn=12854\hss$\nn$\kern1.4\wd0\hss}\hbox to
  0pt{\hss$\mid$\hss}\lower.9\ht0 \hbox to 0pt{\hss$\smile$\hss}\kern\wd0}
\def\ind{\mathop{\mathpalette\Ind{}}}
\def\nind{\mathop{\mathpalette\Notind{}}}

\newcommand{\indt}{\ind^{\!\!\textnormal{\textthorn}}}

\newcommand{\inda}{\ind^{\!\!\textnormal{a}}}

\newcommand{\indd}{\ind^{\!\!\textnormal{d}}}

\newcommand{\indf}{\ind^{\!\!\textnormal{f}}}

\newcommand{\indK}{\ind^{\!\!\textnormal{K}}}

\newcommand{\indb}{\ind^{\!\!\textnormal{b}}}

\newcommand{\indbu}{\ind^{\!\!\textnormal{bu}}}
\newcommand{\nindbu}{\nind^{\!\!\textnormal{bu}}}
\newcommand{\indast}{\ind^{\!\!\ast}}

\newcommand{\Mb}{\mathbb{M}}
\newcommand{\Tc}{\mathcal{T}}
\newcommand{\Fc}{\mathcal{F}}
\newcommand{\Lc}{\mathcal{L}}

\DeclareMathOperator{\acl}{acl}
\DeclareMathOperator{\dcl}{dcl}
\newcommand{\dclu}{\dcl^{\textnormal{u}}}
\DeclareMathOperator{\bdd}{bdd}
\newcommand{\bddu}{\bdd^{\textnormal{u}}}
\DeclareMathOperator{\Aut}{Aut}
\DeclareMathOperator{\Autf}{Autf}
\DeclareMathOperator{\cb}{cb}

\newcommand{\res}{{\upharpoonright}}

\DeclareMathOperator{\tp}{tp}
\DeclareMathOperator{\stp}{stp}

\newcommand{\To}{\Rightarrow}

\DeclareMathOperator{\dom}{dom}

\begin{document}

\title[Bounded ultraimaginary independence]{Bounded ultraimaginary independence and its total Morley sequences}
\address{Department of Mathematics\\
  University of Maryland\\
  College Park, MD 20742, USA}
\author{James Hanson}
\email{jhanson9@umd.edu}
\date{\today}


\begin{abstract}
 We investigate the following model-theoretic independence relation: $b \indbu_A c$ if and only if $\bddu(Ab)\cap \bddu(Ac) = \bddu(A)$, where $\bddu(X)$ is the class of all ultraimaginaries bounded over $X$. In particular, we sharpen a result of Wagner to show that $b \indbu_A c$ if and only if $\langle \Autf(\Mb/Ab)\cup\Autf(\Mb/Ac) \rangle = \Autf(\Mb/A)$, and we establish full existence over hyperimaginary parameters (i.e., for any set of hyperimaginaries $A$ and ultraimaginaries $b$ and $c$, there is a $b' \equiv_A b$ such that $b' \indbu_A c$). Extension then follows as an immediate corollary.

 We also study \emph{total $\indbu$-Morley sequences} (i.e., $A$-indiscernible sequences $I$ satisfying $J \indbu_A K$ for any $J$ and $K$ with $J + K \equiv^{\EM}_A I$), and we prove that an $A$-indiscernible sequence $I$ is a total $\indbu$-Morley sequence over $A$ if and only if whenever $I$ and $I'$ have the same Lascar strong type over $A$, $I$ and $I'$ are related by the transitive, symmetric closure of the relation `$J+K$ is $A$-indiscernible.' This is also equivalent to $I$ being `based on' $A$ in a sense defined by Shelah in his early study of simple unstable theories \cite{Shelah1980}.

  Finally, we show that for any $A$ and $b$ in any theory $T$, if there is an Erd\"os cardinal $\kappa(\alpha)$ with $|Ab|+|T| < \kappa(\alpha)$, then there is a total $\indbu$-Morley sequence $(b_i)_{i<\omega}$ over $A$ with $b_0 = b$. 
\end{abstract}

\maketitle

\section*{Introduction}

\noindent A central theme in neostability theory is the importance of various kinds of `generic' indiscernible sequences---usually with Michael Morley's name attached to them---such as Morley sequences in stable and simple theories, strict Morley sequences in NIP and NTP$_2$ theories, tree Morley sequences in NSOP$_1$ theories, and $\indt$-Morley sequences in rosy theories. A very broad question one might ask is this: How generically can we build indiscernible sequences in \emph{arbitrary} theories? 

Over a model $M$, we can always extend a given type $p(x) \in S_x(M)$ to a global $M$-invariant type $q(x) \supset p(x)$ and then use this to generate a sequence $(b_i)_{i<\omega}$ satisfying $b_i \models q \res M b_{<i}$ for each $i<\omega$. In some cases the particular choice of $q(x)$ matters, but typically these sequences are robustly generic. Sequences produced in this way have a certain property, which is that they are \emph{based on $M$ in the sense of Simon}; i.e., for any $I$ and $J$ with $I\equiv_M J \equiv_M b_{<\omega}$, there is a $K$ such that $I+K$ and $J+K$ are both $M$-indiscernible. In NIP theories, the sequences with this property are precisely the sequences generated by an invariant type \cite[Prop.~2.38]{NIPGuide}. Over an arbitrary set of parameters $A$, however, there may fail to be any indiscernible sequences based on $A$. In the dense circular order, for instance, there are no indiscernible sequences based on $\varnothing$. Other technical issues arise when working over arbitrary sets, as well, such as the necessity of considering Lascar strong types over and above ordinary types.

A notion of independence $\indast$ is said to satisfy \emph{full existence} if for any $A$, $b$, and $c$, there is a $b' \equiv_A b$ such that $b' \indast_A c$. Together with a common model-theoretic application of the Erd\"os-Rado theorem (\cref{fact:ER-ind-seq}), this implies that for any $A$ and $b$, one can build an \emph{$\indast$-Morley sequence}, an $A$-indiscernible sequence $(b_i)_{i<\omega}$ with $b_0 = b$ satisfying $b_i \indast_A b_{<i}$ for each $i<\omega$ (assuming $\indast$ also satisfies right monotonicity). Model-theoretically tame theories often have full existence for powerful independence notions, such as non-forking, but this does fail in some notable tame contexts.

One independence notion that is known to satisfy full existence in arbitrary theories is that of \emph{algebraic independence} \cite[Prop.~1.5]{AdlerGeoIntro}: $b \inda_A c$ if $\acl^{\eq}(Ab) \cap \acl^{\eq}(Ac) = \acl^{\eq}(A)$. A natural modification of this concept is \emph{bounded hyperimaginary independence}: $b \indb_A c$ if $\bdd^{\heq}(Ab)\cap \bdd^{\heq}(Ac) = \bdd^{\heq}(A)$. Despite perhaps sounding like an intro-to-model-theory exercise, the combinatorics necessary to prove full existence for $\inda$ are somewhat subtle. It was recently established in \cite{conant2021separation} by Conant and the author that $\inda$ satisfies full existence in continuous logic and, relatedly, that $\indb$ satisfies full existence in discrete (and continuous) logic, answering a question of Adler \cite[Quest.~A.8]{Adler2005ExplanationOI}. While the relations of $\inda$ and $\indb$ are algebraically nice,\footnote{In the sense of the algebra of an independence relation, not the sense of the algebra in `algebraic closure.'} they seem to lack semantic consequences outside of certain special theories (such as those with a canonical independence relation in the sense of Adler \cite[Lem.~3.2]{Adler2005ExplanationOI}).

While being able to build $\indast$-Morley sequences is certainly good, in many applications the important property is really that of being a \emph{total $\indast$-Morley sequence},\footnote{This use of the term `total' in the context of Morley sequences was originally introduced in \cite{KaplanRamseyOnKim}.} which is an $A$-indiscernible sequence satisfying $b_{\geq i} \indast_A b_{<i}$ for every $i<\omega$. When $\indast$ lacks the algebraic properties necessary to imply that all $\indast$-Morley sequences are total $\indast$-Morley sequences,  it can in general be difficult to ensure their existence. Total $\inda$-Morley sequences arise in Adler's characterization of canonical independence relations. And building total $\indK$-Morley sequences, where $\indK$ is the relation of non-Kim-forking, is a crucial technical step in Kaplan and Ramsey's proofs of the symmetry of Kim-forking and the independence theorem in NSOP$_1$ theories \cite{KaplanRamseyOnKim}.

In simple theories, Morley sequences over $A$ are not generally based on $A$ in the sense of Simon. They do however nearly satisfy this property. If $I$ and $J$ are Morley sequences over $A$ with $I \equiv^\Las_A J$,\footnote{The equivalence relation $\equiv^\Las_A$ is the transitive closure of the relation `there is a model $M \supseteq A$ such that $b \equiv_M c$.' If $b \equiv^\Las_A c$, we say that $b$ and $c$ have the same \emph{Lascar strong type} over $A$.} then there are $I'$ and $K$ such that $I+I'$, $I' +K$, and $J+K$ are $A$-indiscernible. In an NSOP$_1$ theory $T$, if $I$ is a tree Morley sequence over $M\models T$ and $J \equiv_M I$, then we can find $K_0$, $K_1$, and $K_2$ such that $I+K_0$, $K_1 + K_0$, $K_1 + K_2$, and $J + K_2$ are all $M$-indiscernible (see \cref{prop:total-ind-bu-Morley-seq-in-NSOP1}). This fact suggests the consideration of the following equivalence relation, originally introduced by Shelah in \cite[Def.~5.1]{Shelah1980}: Let $\approx_A$ be the transitive, symmetric closure of the relation `$I+J$ is $A$-indiscernible.' The intuition is that what it means for an $A$-indiscernible sequence $I$ to be `based on $A$' is that there are few $\approx_A$-classes among the realizations of $\tp(I/A)$. We say that $I$ is \emph{based on $A$ in the sense of Shelah} if there does not exist a sequence $(I_i)_{i<\kappa}$ (with $\kappa$ large) such that $I_i \equiv_A I$ for each $i<\kappa$ and $I_i \not \approx_A I_j$ for each $i<j<\kappa$. A simple compactness argument shows that $I$ is based on $A$ in the sense of Shelah if and only if the set of realizations of $\tp(I/A)$ decomposes into a bounded number of $\approx_A$-classes. In \cite[Def.~2.4]{Bue97},\footnote{This preprint is difficult to track down. The relevant ideas are developed further by Adler in \cite[Sec.~3.2]{Adler2005ExplanationOI}, which is easily available.} Buechler used this relation to define a notion of canonical base. He focuses on $\varnothing$-indiscernible sequences and gives the following definition: $A$ is a \emph{canonical base} of the $\varnothing$-indiscernible sequence $I$ if any automorphism $\sigma \in \Aut(\Mb)$ fixes $A$ pointwise if and only if it fixes the $\approx_\varnothing$-class of $I$. One difficulty with this concept, of course, is that not all indiscernible sequences have canonical bases in this sense (even in $T^{\eq}$, e.g., \cite[Ex.~3.13]{Adler2005ExplanationOI}).

Two of the problems we have mentioned---the lack of canonical bases for indiscernible sequences and the lack of semantic consequences of $\inda$ and $\indb$---can both be solved by an extremely blunt move: the introduction of ultraimaginary parameters. An \emph{ultraimaginary} is an equivalence class of an arbitrary invariant equivalence relation (as opposed to a type-definable equivalence relation, as in the definition of hyperimaginaries). Every indiscernible sequence $I$ trivially has an ultraimaginary canonical base in the sense of Buechler, i.e., the $\approx_\varnothing$-class of $I$ itself. 

Another appealing aspect of ultraimaginaries is that they characterize Lascar strong type in the same way that hyperimaginaries characterize Kim-Pillay strong type. An ultraimaginary $[b]_E$ is \emph{bounded over $A$} if it has boundedly many conjugates under $\Aut(\Mb/A)$. We will write $\bddu(A)$ for the class of ultraimaginaries bounded over $A$. In general, it turns out that $b$ and $c$ have the same Lascar strong type over $A$ if and only if they `have the same type over $\bddu(A)$,' once this concept is defined precisely.

Pure analogical thinking might lead one to consider the following independence notion: $b \indbu_A c$ if $\bddu(Ab)\cap \bddu(Ac) = \bddu(A)$. This notion is implicit in a result of Wagner \cite[Prop.~2.12]{PLUS-ULTRA}, which we restate and expand slightly (\cref{prop:ind-bu-char}): $b \indbu_A c$ if and only if $\langle \Autf(\Mb/Ab)\cup\Autf(\Mb/Ac) \rangle = \Autf(\Mb/A)$ (where $\langle X \rangle$ is the group generated by $X$). This characterization is clearly semantically meaningful, and moreover it allows one to discuss $\indbu$ without actually mentioning ultraimaginaries at all. One way to see why this equivalence works is the fact that ultraimaginaries are `dual' to co-small sets of automorphisms; a group $G \leq \Aut(\Mb)$ is \emph{co-small} if there is a small model $M$ such that $\Aut(\Mb/M) \leq G$. For every co-small group $G$, there is an ultraimaginary $a_E$ such that $\Aut(\Mb/a_E) = G$ (\cref{prop:co-small-ultra}). 

As $\indbu$ lacks finite character, \emph{total $\indbu$-Morley sequences over $A$} seem to be correctly defined as $A$-indiscernible sequences $(b_i)_{i<\omega}$ with the property that for any $I+J \equiv^{\EM}_A b_{<\omega}$,\footnote{$I \equiv^{\EM}_A J$ means that $I$ and $J$ have the same \emph{Ehrenfeucht-Mostowski type} over $A$ (i.e., for any increasing tuples $\bar{b} \in I$ and $\bar{c} \in J$ of the same length, $\bar{b} \equiv_A \bar{c}$). Note that $I$ and $J$ do not need to have the same order type.} we have that $I \indbu_A J$. The automorphism group characterization of $\indbu$, together with its the nice algebraic properties and the malleability of indiscernible sequences, leads to a pleasing characterization of total $\indbu$-Morley sequences over sets of hyperimaginary parameters (\cref{thm:ind-bu-Morley-char}), the equivalence of the following.
\begin{itemize}
\item $(b_i)_{i<\omega}$ is a total $\indbu$-Morley sequence over $A$.
\item For some infinite $I$ and $J$, we have that $I+J \equiv^{\EM}_A b_{<\omega}$ and $I \indbu_A J$.
\item For any $I$, $I \approx_A b_{<\omega}$ if and only if there is $I' \equiv^\Las_A I$ such that $b_{<\omega} + I'$ is $A$-indiscernible.
\item $b_{<\omega}$ is based on $A$ in the sense of Shelah; i.e., $[b_{<\omega}]_{\approx_A} \in \bddu(A)$.
\end{itemize}
The condition in the third bullet point is a natural mutual generalization of Lascar strong type and Ehrenfeucht-Mostowski type (\cref{defn:LEM-type}). \cref{thm:ind-bu-Morley-char} also tells us that when total $\indbu$-Morley sequences exist, they act as particularly uniform witnesses of Lascar strong type (\cref{prop:Lstp-strong-witness}). 

Of course this all leave two critical questions: Does $\indbu$ always satisfy full existence? And, even if it does, can we actually build total $\indbu$-Morley sequences in any type over any set under any theory? The bluntness of ultraimaginaries leaves us without one of the most important tools in model theory, compactness. Furthermore, $\indbu$'s lack of finite character gives us less leeway in applying the Erd\"os-Rado theorem to construct indiscernible sequences with certain properties; we now need to be more concerned with the particular order types of the sequences involved.

Using some of the indiscernible tree technology from \cite{KaplanRamseyOnKim}, we are able to prove that $\indbu$ does satisfy full existence over arbitrary sets of (hyperimaginary) parameters in arbitrary (discrete or continuous) theories (\cref{thm:full-existence}).\footnote{Although this result partially supersedes a result in \cite{conant2021separation} (full existence for $\inda$ in continuous logic and $\indb$ in discrete or continuous logic), the proof there gives more detailed numerical information which may be especially useful in the metric context.} With regards to building total $\indbu$-Morley sequences, \cref{thm:ind-bu-Morley-char} tells us that we don't need to worry \emph{too} much about order types. All we need to get a total $\indbu$-Morley sequence over $A$ is an $A$-indiscernible sequence $(b_i)_{i<\omega+\omega}$ with $b_{\geq \omega} \indbu_A b_{<\omega}$. This is fortunate because constructing ill-ordered $\indbu$-Morley sequences directly seems daunting. Unfortunately, $\omega+\omega$ appears to be about one $\omega$ further than we can go without a large cardinal. What we do get is this (\cref{thm:Erdos-total-ind-bu-Morley-seq}):  For any $A$ and $b$ in any theory $T$, if there is an Erd\"os cardinal $\kappa(\alpha)$ with $|Ab|+|T| < \kappa(\alpha)$ (for any $\alpha \geq \omega$), then there is a total $\indbu$-Morley sequence $(b_i)_{i<\omega}$ over $A$ with $b_0 = b$. Without a large cardinal, the best we seem to be able to do (\cref{prop:weak-ind-bu-Morley-tree-implies-weakly-total-ind-bu-Morley-seq}) is a half-infinite, half-arbitrary-finite approximation of a total $\indbu$-Morley sequence, which we call a \emph{weakly total $\indbu$-Morley sequence}. These sequences also serve as uniform witnesses of Lascar strong type, without any set-theoretic hypotheses (\cref{cor:Lascar-strong-type-clean-witness}).
\section{Ultraimaginaries}

\noindent Here we will set definitions and conventions, and we also take the opportunity to collect some basic facts about ultraimaginaries which are likely folklore, although we could not find explicit references. 

Fix a theory $T$ and a monster model $\Mb \models T$. 

\begin{defn}
  An \emph{invariant equivalence relation of arity $\kappa$} is an equivalence relation $E$ on $\Mb^x$ (with $|x|=\kappa$) such that for any $a,b,c,d \in \Mb^x$ with $ab \equiv cd$, $aEb$ if and only if $cEd$.

  An \emph{ultraimaginary of arity $\kappa$} is an $E$-equivalence class $a_E$ of some tuple $a \in \Mb^x$ under an invariant equivalence relation $E$ of arity $\kappa$. We may also write such an equivalence class as $[a]_E$ if necessary for notational clarity.

  Given an ultraimaginary $a_E$, $\Aut(\Mb/a_E)$ is the set of automorphisms $\sigma \in \Aut(\Mb)$ with the property that $a E (\sigma\cdot a)$. We write $\Autf(\Mb/a_E)$ for the group generated by $\{\sigma \in \Aut(\Mb/M) : M\preceq \Mb,~\Aut(\Mb/M)\leq \Aut(\Mb/a_E)\}$.  

  We say that $b_F$ is \emph{definable over $a_E$} if $b_F$ is fixed by every automorphism of $\Aut(\Mb/a_E)$. We write $\dclu(a_E)$ for the class of all ultraimaginaries definable over $a_E$. For any $\kappa$, we write $\dclu_\kappa(a_E)$ for the set of elements of $\dclu(a_E)$ of arity at most $\kappa$. We say that $b_F$ and $c_G$ are \emph{interdefinable over $a_E$} if $b_F \in \dclu(a_Ec_G)$ and $c_G \in \dclu(a_Eb_F)$. 

  We say that $b_F$ is \emph{bounded over $a_E$} if the $\Aut(\Mb/b_F)$-orbit of $a_E$ is bounded. We write $\bddu(a_E)$ for the class of all ultraimaginaries bounded over $a_E$. We write $\bddu_\kappa(a_E)$ for the set of elements of $\bddu(a_E)$ of arity at most $\kappa$. We say that $b_F$ and $c_G$ are \emph{interbounded over $a_E$} if $b_F \in \bddu(a_Ec_G)$ and $c_G \in \bddu(a_Eb_F)$.

  We write $a_E \equiv b_E$ to mean that there is an automorphism $\sigma \in \Aut(\Mb)$ with $\sigma\cdot a_E = b_E$. We write $b_F\equiv_{a_E}c_F$ to mean that $a_Eb_F \equiv a_Ec_F$ (i.e., there is $\sigma \in \Aut(\Mb/a_E)$ such that $\sigma \cdot b_F = c_F$).
\end{defn}

We will also sometimes define an invariant equivalence relation $E$ on the realizations of a single type $p(x)$ over $\varnothing$. Equivalence classes of such can be thought of as ultraimaginaries by using the same trick that is commonly used with hyperimaginaries: Consider the invariant equivalence relation $E'(x,y)$ defined by $(x=y\wedge x\not \models p)\vee (E(x,y)\wedge x \models p\wedge y \models p)$.

For the sake of clarity, we will reserve the notation $a_E$ for ultraimaginaries and write hyperimaginaries in the same way we write real elements. For the sake of cardinality issues, we will also take all hyperimaginaries to be quotients of countable tuples by countably type-definable equivalence relations. It is a standard fact that every hyperimaginary is interdefinable with some set of hyperimaginaries of this form. Both of these conventions can be justified more rigorously by passing to the `continuous version of $T^{\eq}$.' See the discussion in \cite[Sec.~4.3]{conant2021separation} for details.

\begin{fact}[{\cite[Lem.~1.2]{SimpInCATs}}]\label{fact:ER-ind-seq}
  Let $(b_i)_{i<\lambda}$ be a sequence of tuples with $|b_i| < \kappa$ and let $A$ be some set of parameters. If $\lambda \geq \beth_{(2^{\kappa+|A|+|T|})^+}$, then there is an $A$-indiscernible sequence $(b'_i)_{i<\omega}$ such that for every $n<\omega$, there are $i_0<\dots < i_n < \kappa$ such that $b'_0\dots b'_n \equiv_A b_{i_0}\dots b_{i_n}$. 
\end{fact}

\begin{lem}\label{lem:bdd-is-dcl-over-models}
  Let $M$ be a model. If $a_E \in \bddu(M)$, then $a_E \in \dclu(M)$.
\end{lem}
\begin{proof}
  Assume that $a_E \notin \dclu(M)$. Let $p(x)$ be a global $M$-invariant type extending $\tp(a/M)$. Assume that there are $a_0$ and $a_1$ realizing $\tp(a/M)$ such that $a_0 \slashed{E} a_1$. For any $i>1$, given $a_{<i}$, let $a_i \models p \res M a_{<i}$. Since $a_ia_j\equiv_M a_i a_k$ for any $j,k<i$, we must have that $a_i \slashed E a_j$ for any $j<i$. Since we can do this indefinitely, we have that $a_E$ is not bounded over $M$.
\end{proof}

\begin{prop}\label{prop:explicit-bound}
  For any ultraimaginaries $a_E$ and $b_F$, the following are equivalent.
  \begin{enumerate}[(1)]
  \item $b_F \notin \bddu(a_E)$.
  \item There is an $a$-indiscernible sequence $(b_i)_{i<\omega}$ such that $b_0 \equiv_{a_E}b$ and $b_i \slashed F b_j$ for each $i<j<\omega$.
  \item $|\Aut(\Mb/a_E)\cdot b_F| > 2^{|ab|+|T|}$.
  \end{enumerate}
\end{prop}
\begin{proof}
  (1)$\To$(2). This follows from an immediate application of \cref{fact:ER-ind-seq}.

  (2)$\To$(3). This follows by stretching the relevant indiscernible sequence to a length of $(2^{|ab|+|T|})^+$.

  (3)$\To$(1). Let $(b^i_F)_{i<(2^{|ab|+|T|})^+}$ be an enumeration of $\Aut(\Mb/a_E)\cdot b_F$. Let $M \supseteq a$ be a model with $|M| \leq |a|+|T|$. Let $x$ be a tuple of variables of the same length as $b$. There are at most $2^{|ab|+|T|}$ types in $S_x(M)$. Therefore, there must be $i<j<(2^{|ab|+|T|})^+$ such that $b^i\equiv_M b^j$. Let $p(x)$ be a global $M$-invariant type extending $\tp(b^i/M)$, and let $(c^i)_{i<\lambda}$ be a Morley sequence generated by $p(x)$ over $Mb^ib^j$. Since $b^i\slashed{F} b^j$, we must have that $c^0 \slashed{F} b^i$. Therefore $c^i \slashed{F} c^j$ for any $i<j<\lambda$. Furthermore, since $c^i \equiv_M c^j$, we also have that $c^i\equiv_a c^j$ and therefore also $c^i_F \equiv_{a_E}c^j_F\equiv_{a_E}b_F$. Since we can do this for any $\lambda$, we have that $b_F \notin \bddu(a_E)$.
\end{proof}

\begin{cor}
  For any $\lambda$, $\bddu_\lambda(a_E)$ has cardinality at most $2^{|a| + 2^{\lambda+|T|}}$.
\end{cor}
\begin{proof}
  There are at most $2^{2^{\lambda+|T|}}$ many invariant equivalence relations of arity at most $\lambda$. For each such $F$, the set $\{b_F : b_F \in \bddu_\lambda(a_E)\}$ has cardinality at most $2^{|a|+\lambda+|T|}$ by \cref{prop:explicit-bound}. Finally, $2^{2^{\lambda+|T|}}\cdot  2^{|a|+\lambda+|T|} = 2^{|a|+2^{\lambda+|T|}}$.
\end{proof}

\subsection{Co-small groups of automorphisms}

\noindent Here we will see that ultraimaginaries are essentially the same thing as reasonable subgroups of $\Aut(\Mb)$.

\begin{defn}
  A set $G \leq \Aut(\Mb)$ is \emph{co-small} if there is a small model $M$ such that $\Aut(\Mb/M) \leq G$.
\end{defn}

Clearly for any ultraimaginary $a_E$, $\Aut(\Mb/a_E)$ is co-small. The converse is true as well.

\begin{prop}\label{prop:co-small-ultra}
  For any co-small $G$, if $\Aut(\Mb/M)\leq G$, then there is an ultraimaginary $a_E$ such that $G = \Aut(\Mb/a_E)$ where $a$ is some enumeration of $M$. 
\end{prop}
\begin{proof}
  Let $M$ be a small model witnessing that $G$ is co-small. Consider the binary relation defined on realizations of $\tp(M)$ (in some fixed enumeration) defined by $E(M_0,M_1)$ if and only if there is $\sigma \in \Aut(\Mb)$ and $\tau \in G$ such that $\sigma \cdot M = M_0$ and $\sigma\tau\cdot M = M_1$. We need to verify that $E$ is an invariant equivalence relation. Reflexivity is obvious.

  \vspace{0.5em}
\noindent  \emph{Invariance.} Suppose that $E(M_0,M_1)$, as witnessed by $\sigma \in \Aut(\Mb)$ and $\tau \in G$. Fix $\sigma' \in \Aut(\Mb)$. We then have that $\sigma'\sigma\cdot M = \sigma'\cdot M_0$ and $\sigma'\sigma\tau\cdot M = \sigma' \cdot M_11$, whence $E(\sigma'\cdot M_0,\sigma'\cdot M_1)$.

\vspace{0.5em}
\noindent \emph{Symmetry.} If $\sigma \cdot M = M_0$ and $\sigma\tau \cdot M = M_1$ with $\sigma \in \Aut(\Mb)$ and $\tau \in G$, then $\sigma\tau\tau^{-1}\cdot M = M_0$ and $\sigma\tau\cdot M = M_1$. We have $\sigma\tau \in \Aut(\Mb)$ and $\tau^{-1}\in G$, so $E(M_1,M_0)$.

\vspace{0.5em}
\noindent \emph{Transitivity.} Suppose that for $\sigma,\sigma' \in \Aut(\Mb)$ and $\tau,\tau'\in G$, we have that $\sigma\cdot M = M_0$, $\sigma\tau\cdot M = \sigma'\cdot M = M_1$, and $\sigma'\tau'\cdot M = M_2$. This implies that $(\sigma\tau)^{-1}\sigma' = \tau^{-1}\sigma^{-1}\sigma' \in \Aut(\Mb/M)\leq  G$. Since $\tau \in G$ as well, we have that $\sigma^{-1}\sigma' \in G$. Therefore $\sigma^{-1}\sigma'\tau' \in G$. Finally, $\sigma\sigma^{-1}\sigma'\tau' \cdot M = M_2$, so $E(M_0,M_2)$. 

\vspace{0.5em}
Consider the ultraimaginary $M_E$.  For any $\tau \in G$, we clearly have $E(M,\tau\cdot M)$, so $G \leq \Aut(\Mb/M_E)$. Conversely, suppose that $\alpha \in \Aut(\Mb/M_E)$. By definition, this implies that $E(M,\alpha\cdot M)$, so there are $\sigma \in \Aut(\Mb)$ and $\tau \in G$ such that $\sigma \cdot M = M$ and $\sigma\tau\cdot M = \alpha \cdot M$. Therefore $\sigma, \tau^{-1}\sigma^{-1}\alpha \in \Aut(\Mb/M) \leq G$. Since $\tau^{-1} \in G$, we therefore have that $\alpha \in G$. 
\end{proof}

\begin{cor}\label{cor:shrink-ultra}
  If $b_F \in \bddu(a_E)$, then there is $c_G \in \bddu(a_E)$ of arity at most $|a|+|T|$ such that $b_F$ and $c_G$ are interdefinable over $\varnothing$. Furthermore, $c$ can be taken to be an enumeration of any model of size at most $|a|+|T|$ containing $a$.
\end{cor}
\begin{proof}
  There is a model $M \supseteq a$ with $|M| \leq |a|+|T|$. By \cref{lem:bdd-is-dcl-over-models}, we have that $\Aut(\Mb/M) \leq \Aut(\Mb/b_F)$, so by \cref{prop:co-small-ultra}, we have that there is $c_G$ with arity at most $|a|+|T|$ which satisfies that $\Aut(\Mb/c_G) = \Aut(\Mb/b_F)$ (i.e., $c_G$ and $b_F$ are interdefinable over $\varnothing$). Furthermore, we can take $c$ to be an enumeration of $M$.
\end{proof}

\begin{defn}
  For any co-small group $G$, we write $\tbr{G}$ for some arbitrary ultraimaginary $a_E$ of minimal arity satisfying $G = \Aut(\Mb/a_E)$. We will write $\dclu\tbr{G}$ and $\dclu_\lambda\tbr{G}$ for $\dclu(\tbr{G})$ and $\dclu_\lambda(\tbr{G})$ and likewise with $\bddu$. (Note that $\dclu\tbr{G}$ and $\bddu\tbr{G}$ only depend on $G$, not on the particular choice of $\tbr{G}$.)
\end{defn}

It is immediate from \cref{prop:co-small-ultra} that for any co-small $G$ and $H$, $\tbr{G} \in \dclu\tbr{H}$ if and only if $G \geq H$. A similar statement is true of $\bddu$ (\cref{prop:Gf-group-char}).

Now we can see that intersections of $\dclu$-closed sets (and therefore also $\bddu$-closed sets) have semantic significance in arbitrary theories, in that intersections correspond to joins in the lattice of co-small groups of automorphisms. 

\begin{prop}\label{prop:crab-walk-char}
  For any $a_E$, $b_F$, $c_G$, and $c'_G$, the following are equivalent.
  \begin{enumerate}[(1)]
  \item $c_G \equiv_{\dclu_\lambda(a_E) \cap \dclu_\lambda(b_F)} c'_G$ for all $\lambda$.
  \item There is $\sigma \in \langle \Aut(\Mb/a_E)\cup\Aut(\Mb/b_F) \rangle$ such that $\sigma \cdot c_G = c'_G$.
  \item There is a sequence $(a^ib^ic^i)_{i \leq n}$ such that $a^0 = a$, $b^0 = b$, $c^0 =c$, $c^n = c'$, and for each $i<n$,
    \begin{itemize}
    \item if $i$ is even, then $a^i = a^{i+1}$ and $b^i_Fc^i_G \equiv_{a^i_E} b^{i+1}_Fc^{i+1}_G$ and
    \item if $i$ is odd, then $b^i = b^{i+1}$ and $a^i_E c^i_G \equiv_{b^i_F} a^{i+1}_Ec^{i+1}_G$.
    \end{itemize}
  \end{enumerate}
\end{prop}
\begin{proof}
  Let $H = \langle \Aut(\Mb/a_E) \cup \Aut(\Mb/b_F) \rangle$.
  \vspace{0.5em}
  
  \noindent  \emph{Claim.} $\dclu_\lambda(a_E)\cap \dclu_\lambda(b_F)$ and $\tbr{H}$ are interdefinable for sufficiently large $\lambda$.
  \vspace{0.25em}

  \noindent  \emph{Proof of claim.} Clearly $\tbr{H} \in \dclu(a_E)\cap \dclu(b_F)$, so $\tbr{H} \in \dclu_\lambda(a_E)\cap \dclu_\lambda(b_F)$ for all sufficiently large $\lambda$.

  Conversely, suppose that $d_I \in \dclu(a_E)\cap \dclu(b_F)$. Any $\sigma \in H$ is a product of elements of $\Aut(\Mb/a_E)$ and $\Aut(\Mb/b_F)$, so it must fix $d_I$. Therefore $\Aut(\Mb/d_I)\geq H$ and hence $d_I \in \dclu\tbr{H}$. \hfill $\square_{\text{claim}}$

  \vspace{0.5em}
  So now we have that $c_G \equiv_{\dclu_\lambda(a_E)\cap \dclu_\lambda(b_F)}c'_G$ holds for sufficiently large $\lambda$ if and only if $c_G \equiv_{\tbr{H}}c'_G$. Also note that $c_G \equiv_{\dclu_\lambda(a_E) \cap \dclu_\lambda(b_F)}c'_G$ for sufficiently large $\lambda$ and only if the same holds for any $\lambda$. Therefore we have that (1) and (2) are equivalent.

  There is a $\sigma \in H$ with $\sigma \cdot c_G = c'_G$ if and only if there are $\alpha_0,\dots,\alpha_n \in \Aut(\Mb/a_E)$ and $\beta_0,\dots,\beta_n \in \Aut(\Mb/b_E)$ such that $\sigma = \beta_n\alpha_n\beta_{n-1}\dots \alpha_1\beta_0\alpha_1$. The existence of such $\bar{\alpha}$ and $\bar{\beta}$ for which $\beta_n\alpha_n\beta_{n-1}\dots \alpha_1\beta_0\alpha_1\cdot c_G = c'_G$ is clearly equivalent to (3), so we have that (2) and (3) are equivalent.
\end{proof}

A similar statement is true for intersections of arbitrary families of small sets.

\subsection{Lascar strong type}

\begin{defn}
  For any co-small group $G\leq \Aut(\Mb)$, let $G_f$ be the group generated by all groups of the form $\Aut(\Mb/M)\leq G$ with $M$ a small model. For any ultraimaginary $a_E$, let $\Autf(\Mb/a_E) = \Aut(\Mb/a_E)_f$.

  We say that $b_F$ and $c_F$ have the same \emph{Lascar strong type over $a_E$}, written $b_F \equiv^\Las_{a_E} c_F$, if there is $\sigma \in \Autf(\Mb/a_E)$ such that $\sigma \cdot b_F = c_F$.
\end{defn}

\begin{prop}\label{prop:Gf-group-char}
  For any co-small groups $G$ and $H$, $\tbr{G} \in \bddu\tbr{H}$ if and only if $G \geq H_f$.
\end{prop}
\begin{proof}
  First note that for a model $M$, by \cref{lem:bdd-is-dcl-over-models}, we have that $\tbr{G} \in \bddu(M)$ if and only if $G \geq \Aut(\Mb/M)$. Therefore, for any model $M$ with $\tbr{H} \in \bddu(M)$, we must have that $G \geq \Aut(\Mb/M)$. Since $\tbr{H} \in \bddu(M)$ if and only if $H \geq \Aut(\Mb/M)$, we have that $G \geq H_f$.

  Conversely, assume that $G \geq H_f$. This implies that for any small model $M$ with $\tbr{H} \in \bddu(M)$, we have $H_f \geq \Aut(\Mb/M)$, so $G \geq \Aut(\Mb/M)$ and $\tbr{G} \in \dclu(M)$. Fix some such model $N$. Assume for the sake of contradiction that $\tbr{G} \notin \bddu\tbr{H}$. For any $\lambda$, we can find $(\sigma_i)_{i<\lambda}$ in $H = \Aut(\Mb/\tbr{H})$ such that $\sigma_i \cdot \tbr{G} \neq \sigma_j \cdot \tbr{G}$ for each $i<j<\lambda$. If $\lambda$ is larger than $2^{|N|+|T|}$, there must be $i<j<\lambda$ such that $\sigma_i\cdot \tbr{G} \equiv_N \sigma_j \cdot \tbr{G}$. Let $N' = \sigma_i^{-1}\cdot N$. $N'$ is now a model satisfying $\Aut(\Mb/N') \leq G$. So $\tbr{G} \in \dclu(N')$, but $\tbr{G} \equiv_{N'} \sigma_i^{-1}\sigma_j\cdot \tbr{G}$ and $\tbr{G} \neq \sigma_i^{-1}\sigma_j\cdot \tbr{G}$, which is a contradiction. 
\end{proof}

An important fact about ultraimaginaries is that $\bddu$ has the same relationship with Lascar strong type that $\bdd^{\heq}$ has with Kim-Pillay strong type.

For any $a_E$ and $b_F$, by an abuse of notation, we'll write $[b_F]_{\equiv^\Las_{a_E}}$ for $[ab]_G$, where $G(ab,a'c')$ holds if and only if $aEa'$ and $b\equiv^\Las_{a_E}b'$.

\begin{prop}\label{prop:Lstp-char}
  For any ultraimaginaries $a_E$, $b_F$, and $c_F$, the following are equivalent.
  \begin{enumerate}[(1)]
  \item $b_F \equiv^\Las_{a_E} c_F$.
  \item $b_F \equiv_{\bddu_\lambda(a_E)} c_F$ for all sufficiently large $\lambda$.
  \item $b_F \equiv_{\bddu_{|a|+|T|}(a_E)} c_F$.
  \end{enumerate}
\end{prop}
\begin{proof}
  To see that (1) implies (3), fix $M \supseteq a$ and some automorphism $\sigma \in \Aut(\Mb/M)$. By \cref{lem:bdd-is-dcl-over-models}, we have that $\Aut(\Mb/M)\leq \Aut(\Mb/\bddu_{|a|+|T|}(a_E))$. Therefore  $b_F \equiv_{\bddu_{|a|+|T|}(a_E)} \sigma\cdot b_F$. By induction, we therefore have that $b_F \equiv^\Las_{a_E} c_F$ implies $b_F \equiv_{\bddu_{|a|+|T|}(a_E)}c_F$.

\cref{cor:shrink-ultra} implies that $\Aut(\Mb/\bddu_\lambda(a_E)) \geq \Aut(\Mb/\bddu_{|a|+|T|}(a_E))$ for all $\lambda$, so (3) implies (2).

To see that (2) implies (1), note that $[b_F]_{\equiv^\Las_{a_E}} \in \bddu_{\lambda}(a_E)$ for some sufficiently large $\lambda$. Therefore if $b_F \equiv_{\bddu_\lambda(a_E)} c_F$, we must have $[b_F]_{\equiv^\Las_{a_E}} = [c_F]_{\equiv^\Las_{a_E}}$ or, in other words, $b_F \equiv^\Las_{a_E} c_F$. 
\end{proof}

\section{Bounded ultraimaginary independence}

\begin{defn}
  
  We write $b_F \indbu_{a_E} c_G$ to mean that $\bddu(a_Eb_F)\cap \bddu(a_Ec_G) = \bddu(a_E)$.
\end{defn}

An easy argument shows that if $c_G \in \bddu(b_F)$ and $b_F \in \bddu(a_E)$, then $c_G \in \bddu(a_E)$. Note also that $b_F \indbu_{a_E} c_G$ if and only if $\bddu_\kappa(a_Eb_F)\cap \bddu_\kappa(a_Ec_G) = \bddu_\kappa(a_E)$ for all $\kappa$. $\indbu$ satisfies some of the familiar properties of $\inda$.

\begin{prop} Fix ultraimaginaries $a_E$, $b_F$, $c_G$, and $e_I$. 
  \begin{itemize}
  \item \emph{(Invariance)} If $a_Eb_Fc_G \equiv a'_Eb'_Fc'_G$, then $b_F \indbu_{a_E}c_G$ if and only if $b'_F \indbu_{a'_E}c'_G$.
  \item \emph{(Symmetry)} $b_F \indbu_{a_E} c_G$ if and only if $c_G \indbu_{a_E} b_F$.
  \item \emph{(Monotonicity)} If $b_Fc_G \indbu_{a_E} d_He_I$, then $b_F \indbu_{a_E}d_H$.
  \item \emph{(Transitivity)} If $b_F \indbu_{a_E}c_G$ and $d_H \indbu_{a_Eb_F}c_G$, then $b_Fd_H \indbu_{a_E}c_G$.
  \item \emph{(Normality)} If $b_F \indbu_{a_E}c_G$, then $a_Eb_F \indbu_{a_E}a_Ec_G$.
  \item \emph{(Anti-reflexivity)} If $b_F \indbu_{a_E}b_F$, then $b_F \in \bddu(a_E)$.
  \end{itemize}
\end{prop}
\begin{proof}
  Everything except transitivity is immediate. For transitivity, assume that $b_F \indbu_{a_E}c_G$ and $d_H \indbu_{a_Eb_F}c_G$. Let $e_I$ be an element of $\bddu(a_Eb_Fd_H)\cap \bddu(a_Ec_G)$. This implies that it is an element of $\bddu(a_Eb_Fd_H)\cap \bddu(a_Eb_Fc_G)$, so by assumption it is an element of $\bddu(a_Eb_F)$. But this means that it's in both $\bddu(a_Eb_F)$ and $\bddu(a_Ec_G)$, so by assumption again, it is an element of $\bddu(a_E)$.
\end{proof}

Part of the goal of this paper is to prove full existence and therefore also extension for $\indbu$ (although only over hyperimaginary bases).
\begin{itemize}
\item (Full existence over hyperimaginaries) For any set of hyperimaginaries $A$ and ultraimaginaries $b_E$ and $c_F$, there is $c'_F \equiv_{A} c_F$ such that $b_E \indbu_{A}c'_F$.
\item (Extension over hyperimaginaries) For any set of hyperimaginaries $A$ and ultraimaginaries $b_E$, $c_F$, and $d_G$, if $b_E \indbu_{A}c_F$, then there is $b'_E \equiv_{Ac_F}b_E$ such that $b'_E \indbu_{A}c_Fd_G$. 
\end{itemize}

 Finite character fails very badly, of course. Local character seems unlikely except possibly in the presence of large cardinals. We do have some control over the relevant cardinalities, however.

\begin{prop}\label{prop:card-bound}
  For any $a_E$, $b_F$, and $c_G$, $b_F \indbu_{a_E}c_G$ if and only if $\bddu_{\lambda}(a_Eb_F)\cap \bddu_{\lambda}(a_Ec_G) = \bddu_{\lambda}(a_E)$, where $\lambda = |ab|+|T|$.
\end{prop}
\begin{proof}
  Let $\lambda = |ab|+|T|$. Clearly we have that if $b_F \indbu_{a_E}c_G$, then $\bddu_\lambda(a_Eb_F)\cap \bddu_\lambda(a_Ec_G) = \bddu_\lambda(a_E)$.

  Conversely, assume that $b_F \nindbu_{a_E}c_G$. There is some $d_H \in (\bddu(a_Eb_F)\cap\bddu(a_Ec_G))\setminus \bddu(a_E)$. By \cref{cor:shrink-ultra}, there is $e_I$ of arity at most $\lambda$ such that $d_H$ and $e_I$ are interdefinable. This means that $e_I \in (\bddu_\lambda(a_Eb_F)\cap \bddu_\lambda(a_Ec_G))\setminus \bddu_\lambda(a_E)$. Therefore $\bddu_\lambda(a_E b_F)\cap \bddu_\lambda(a_E c_G) \neq \bddu_\lambda(a_E)$. 
\end{proof}

The following characterization of $\indbu$ (and the manner of proof) is essentially due to Wagner \cite{PLUS-ULTRA}.

\begin{prop}\label{prop:ind-bu-char}
  For any ultraimaginaries $a_E$, $b_F$, and $c_G$, the following are equivalent.
  \begin{enumerate}[(1)]
  \item $b_F \indbu_{a_E}c_G$.
  \item For any $b'_F \equiv^\Las_{a_E} b_F$, there are $b^0,c^0,b^1,c^1,\dots, c^{n-1},b^n$ such that $b^0 = b$, $c^0 = c$, $b^n = b'$, and for each $i<n$, $b^i_F \equiv^\Las_{a_Ec^i_G} b^{i+1}_F$ and $c^{i}_G \equiv_{a_E b^{i+1}_F} c^{i+1}_G$ if $i<n-1$. 
  \item $\langle \Autf(\Mb/a_Eb_F)\cup \Autf(\Mb/a_Ec_G) \rangle = \Autf(\Mb/a_E)$.
  \end{enumerate}
\end{prop}
\begin{proof}
  The equivalence of (1) and (3) follows from Propositions~\ref{prop:crab-walk-char} and \ref{prop:Lstp-char} and the fact that intersections of $\bddu$-closed sets are $\bddu$-closed. These propositions also clearly give that (3) implies (2) (after some re-indexing of the relevant sequences).

  So assume (2), but also assume for the sake of contradiction that (1) fails. Let $d_H$ be an element of $(\bddu(a_E b_F) \cap \bddu(a_Ec_G)) \setminus \bddu(a_E)$. Since $d_H$ is not bounded over $a_E$, there must be some $d'_H \equiv^\Las_{a_E} d_H$ such that $d'_H \notin \bddu(a_Eb_E)\cap \bddu(a_Ec_G)$. Find $b'_F$ such that $b_Fd_H \equiv^\Las_{a_E} b'_Fd'_H$.  Let $b^0,c^0,b^1,c^1\dots,c^{n-1},b^n$ be as in (2), with $b^n = b'$. Find $d^{\sfrac{1}{2}},d^{1},d^{\sfrac{3}{2}},d^{2},\dots,d^{n-\sfrac{1}{2}},d^n$ such that $d^{\sfrac{1}{2}} = d$ and for each $i<n$,
  \begin{itemize}
  \item $b^i_F d^{i+\sfrac{1}{2}}_H \equiv^\Las_{a_E c^i_G} b^{i+1}_F d^{i+1}$ and 
  \item $c^i_G d^{i+1}_H \equiv^\Las_{a_E b^{i+1}_F}c^{i+1}_G d^{i+\sfrac{3}{2}}_H$ if $i<n-1$.
  \end{itemize}
  We now have that $b'_Fd'_H \equiv^\Las_{a_E}b_F d^n_H\equiv^\Las_{a_E} b'_F d^n_H$, so in particular, $d'_H \equiv^\Las_{a_Eb'_F}d^n_H$. But note that for each $i<n$, we have that
  \[
    \bddu(a_Eb^i_F)\cap \bddu(a_Ec^i_G) = \bddu(a_Eb^{i+1}_F)\cap \bddu(a_Ec^i_G)
  \]
  and, if $i<n-1$,
  \[
    \bddu(a_Eb^{i+1}_F)\cap \bddu(a_Ec^i_G) = \bddu(a_Eb^{i+1}_F)\cap \bddu(a_Ec^{i+1}_G).
  \]
  Therefore $d^n_H \in \bddu(a_Eb^n_F)\cap \bddu(a_Ec_G^{n-1})$, so since $d^n_H \equiv^\Las_{a_Eb^n_F} d'_H$, we must also have $d'_H \in \bddu(a_Eb^n_F)\cap \bddu(a_Ec_G^{n-1}) = \bddu(a_Eb_F)\cap\bddu(a_Ec_G)$, which is a contradiction.
\end{proof}

\section{Full existence}
\noindent We will use the tree bookkeeping machinery from \cite{KaplanRamseyOnKim}, with some minor extensions (the notation $\Tc_\alpha^\ast$ and $\Fc_\alpha$).  

\begin{defn}
  For any ordinal $\alpha$, $\Lc_{s,\alpha}$ is the language 
  \[
  \{\tleq,\wedge,<_{lex},P_0,P_1,\dots,P_{\beta}(\beta<\alpha),\dots\},
  \]
   with $\tleq$ and $<_{lex}$ binary relations, $\wedge$ a binary function, and each $P_\beta$ a unary relation.
  
   For any ordinal $\alpha$, we write $\Tc_\alpha^\ast$ for the set of functions $f$ with codomain $\omega$ and finite support such that $\dom(f)$ is an end segment of $\alpha$. We write $\Tc_\alpha$ for the set of functions $f \in \Tc_\alpha^\ast$ with $\dom(f) = [\beta,\alpha)$ and $\beta$ is not a limit ordinal. We write $\Fc_{\alpha+1}$ (for forest) for $\Tc_{\alpha+1}\setminus \{\varnothing\}$. 

  We interpret $\Tc_\alpha^\ast$ and $\Tc_\alpha$ as $\Lc_{s,\alpha}$-structures by
  \begin{itemize}
  \item $f \tleq g$ if and only if $f\subseteq g$;
  \item $f \wedge g = f\res[\beta,\alpha)=g\res [\beta,\alpha)$, where $\beta=\min\{\gamma : f\res [\gamma,\alpha)=g\res [\gamma,\alpha)\}$ (with the understanding that $\min \varnothing = \alpha$);
  \item $f <_{lex} g$ if and only if either $f \lhd g$ or $f$ and $g$ are $\tleq$-incomparable, $\dom(f\wedge g) =[\gamma,\alpha)$, and $f(\gamma) < g(\gamma)$; and
  \item $P_\beta(f)$ holds if and only if $\dom(f) = [\beta,\alpha)$.
  \end{itemize}

  We write $\langle  i  \rangle$ for the element $\{(\alpha,i)\}$ of $\Tc_\alpha^\ast$. For $f \in \Tc_\alpha^\ast$ with $\dom(f) = [\beta+1,\alpha)$ and $i < \delta$, we write $f \frown \langle i\rangle$ for the function $f\cup\{(\beta,i)\}$. In general, we write $\langle i\rangle\frown f$ for the element of $\Tc_{\alpha+1}^\ast$ given by $f \cup \{(\alpha,i)\}$.\footnote{Note that when $\alpha$ is a non-limit ordinal, $\langle i \rangle$ is an element of $\Tc_\alpha$. When $f \in \Tc_\alpha$ and $f \frown \langle i \rangle$ is defined, then $f \frown \langle  i \rangle \in \Tc_\alpha$. Likewise, if $f \in \Tc_{\alpha+1}$, then $\langle i \rangle\frown f \in \Tc_\alpha$.}
  
  For $\alpha < \beta$, we define the canonical inclusion map $\iota_{\alpha\beta} : \Tc_\alpha \to \Tc_\beta$ by $\iota_{\alpha\beta}(f) = f \cup \{(\gamma,0) : \gamma \in \beta \setminus \alpha\}$. (Note that $\iota_{\alpha,\alpha+1}(f)=\langle 0\rangle\frown f$.)

  For $\beta < \alpha$, we write $\zeta_\beta$ for the function whose domain is $[\beta,\alpha)$ with the property that $\zeta_\beta(\gamma) = 0$ for all $\gamma \in [\beta,\alpha)$.
\end{defn}

Given a family $(b_f)_{f \in X}$, we may refer to it briefly as $b_{\in X}$. 

\begin{defn}
  We say that a tree $(b_f)_{f \in \Tc_\alpha}$ is \emph{$s$-indiscernible over $A$} if for any tuples $f_0\dots f_{n-1}$ and $g_0\dots g_{n-1}$ in $\Tc_\alpha$ with $f_0\dots f_{n-1} \equiv^{\mathrm{qf}} g_0\dots g_{n-1}$, $b_{f_0}\dots b_{f_{n-1}}\equiv_A b_{g_0}\dots b_{g_{n-1}}$, where quantifier-free type is in the language $\Lc_{s,\alpha}$. (Note that this does not entail that $b_f$'s on different levels have the same length as tuples.) We also say that $(b_f)_{f \in \Fc_\alpha}$ is $s$-indiscernible over $A$ (for $\alpha$ a non-limit) if it is $s$-indiscernible after adding a root node $b_{\zeta_\alpha} = \varnothing$.
\end{defn}

Given $f \in \Tc_\alpha$, we write $b_{\tgeq f}$ to refer to some fixed enumeration of the set $\{b_g : g \in \Tc_\alpha,~f \tleq g\}$. In particular, we choose this enumeration in a uniform way so that if $(b_f)_{f \in \Tc_\alpha}$ is $s$-indiscernible over $A$, then for any $f$ of successor length, the sequence $(b_{f \frown \langle  i \rangle})_{i<\omega}$ is $A$-indiscernible. When $f$ is an element of $\Tc^\ast_\alpha$, we will also write, by an abuse of notation, $b_{\tgeq f}$ for some fixed enumeration of the set $\{b_g : g \in \Tc_\alpha,~f \subseteq g\}$. One particular example of this will be sequences of the form $(b_{\tgeq \zeta_{\alpha+1}\frown \langle i \rangle})_{i <\omega}$, where $\alpha$ is a limit ordinal. This is essentially the only situation in which we need to consider $\Tc_\alpha^\ast$.

Note that for a limit ordinal $\alpha$, $(b_f)_{f \in \Tc_\alpha}$ is $s$-indiscernible over $A$ if and only if $(b_f)_{f \in \iota_{\beta,\alpha}(\Tc_\beta)}$ is $s$-indiscernible over $A$ for every $\beta < \alpha$.

We will also need the following fact.

\begin{fact}[{Modeling property for $s$-indiscernibles \cite[Thm.~4.3]{KKS2013}}]\label{fact:modeling-for-s-ind}
  Let $X$ be $\Tc_\alpha$ or $\Fc_\alpha$. For any $(b_f)_{f \in X}$ and any set $A$ of \textbf{hyper}imaginaries, there is a family of tuples $(c_f)_{f \in X}$ that is $s$-indiscernible over $A$ and \emph{locally based on $b_{\in X}$} (i.e., for any finite tuple $f_0\dots f_{n-1}$ from $X$ and any neighborhood $U$ of $\tp(c_{f_0}\dots c_{f_{n-1}}/A)$ (in the appropriate type space), there is a tuple $g_0\dots g_{n-1}$ from $X$ such that $f_0\dots f_{n-1} \equiv^{\mathrm{qf}} g_0\dots g_{n-1}$ and $\tp(b_{g_0}\dots b_{g_{n-1}}) \in U$).
\end{fact}

Note that while \cref{fact:modeling-for-s-ind} is normally formulated for discrete logic, the corresponding statement in continuous logic (and therefore also for hyperimaginaries) follows easily from a very soft general argument: Given a metric structure $M$ and a tree $(b_f)_{f \in X}$ of elements of $M$, pass to a discretization of the theory and apply \cite[Thm.~4.3]{KKS2013} there. Then take the resulting $s$-indiscernible family back to the original continuous theory. (See \cite{1905.10673} for details regarding this kind of construction.)

Now we are ready to prove full existence for $\indbu$, but we will take the opportunity to prove a certain technical strengthening which we will need later, in the construction of $\indbu$-Morley trees.

\begin{lem}\label{lem:full-existence-technical}
  If $(b_f)_{f \in \Tc_\alpha}$ is a tree of real elements that is $s$-indiscernible over a set of \textbf{hyper}imaginaries $A$, then there is a $\gamma > \alpha$ and a tree $(e_f)_{f \in \Tc_{\gamma+1}}$ such that
  \begin{itemize}
  \item $e_{\in \Tc_{\gamma+1}}$ is $s$-indiscernible over $A$,
  \item for each $f \in \Tc_\alpha$, $b_f = e_{\iota_{\alpha,\gamma+1}(f)}$, and
  \item $e_{\tgeq \zeta_\alpha} \indbu_A e_{\tgeq h}$, where $\dom(h) = [\alpha,\gamma+1)$, $h(\delta)=0$ for all $\delta \in [\alpha,\gamma)$, and $h(\gamma)=1$. 
  \end{itemize}
  (Note that $e_{\tgeq \zeta_\alpha}$ is the original tree.)
\end{lem}
\begin{proof}
  If $b_{\in \Tc_\alpha} \in \acl(A)$, then the statement is trivial, so assume that $b_{\in \Tc_\alpha}\notin \acl(A)$.
  
  Fix $\lambda = |Ab_{\in \Tc_\alpha}|+|T|$.  By \cref{prop:card-bound}, we have that that $b_{\in \Tc_{\alpha}} \indbu_A c$ if and only if $\bddu_\lambda(Ab_{\in \Tc_{\alpha}})\cap \bddu_\lambda(Ac) = \bddu_\lambda(A)$ for any $c$. Let $\mu = |\bddu_\lambda(Ab_{\in \Tc_{\alpha}})\setminus \bddu_\lambda(A)|$.

  We will build a family $(e_f : f \in \iota_{\gamma+1,\mu^+}(\Tc_{\gamma+1}))$ inductively, where $\gamma$ is some successor ordinal less than $\mu^+$. By an abuse of notation, we will systematically conflate the sets $\iota_{\alpha\mu^+}(\Tc_\alpha)$ and $\Tc_\alpha$ (and likewise for $\iota_{\alpha\mu^+}(\Fc_{\alpha+1})$ and $\Fc_{\alpha+1}$) for all $\alpha <\mu^+$.\footnote{Note that the $\zeta_\beta$ and $f\frown \langle i\rangle$ notations are consistent with this conflation but also that the $\langle i \rangle\frown f$ notation is \emph{not}.} Note that in general this will mean that $e_{\tgeq \zeta_\beta}$ is the same thing as $e_{\in \Tc_{\beta}}$. 

  Let $e_{f}=b_f$ for all $f \in \Tc_\alpha$. Since $b_{\in \Tc_\alpha} \notin \acl(A)$, we can find a family $(d_{f})_{f \in \Fc_{\alpha+1}}$ extending $e_{\in \Tc_\alpha}$ such that $(d_{\tgeq \zeta_{\alpha+1}\frown \langle i \rangle})_{i<\omega}$ is a non-constant $A$-indiscernible sequence. By \cref{fact:modeling-for-s-ind}, we can define $e_f$ for all $f \in \Fc_{\alpha+1}$ in such a way that the family $e_{\in \Fc_{\alpha+1}}$ is locally based on $d_{\in \Fc_{\alpha+1}}$. In particular, $(e_{\tgeq \zeta_{\alpha+1}\frown \langle i \rangle})_{i < \omega}$ will be a non-constant $A$-indiscernible sequence.

  At successor stage $\beta+1 \geq \alpha$, assume that we have defined $e_{f}$ for all $f \in \Fc_{\beta+1}$ and that the family $(e_f)_{f \in \Fc_{\beta+1}}$ is $s$-indiscernible over $A$. If there is no $d_E \in \bddu_\lambda(Ab_{\in \Tc_\alpha})\setminus \bddu_\lambda(A)$ such that the family $(e_f)_{f \in \Fc_{\beta+1}}$ is $s$-indiscernible over $Ad$, let $e_{\zeta_{\beta+1}} = \varnothing$ and $\gamma = \beta$ and halt the construction. Otherwise, let $e_{\zeta_{\beta+1}} = d$. For later reference, let $E_{\beta+1}$ be $E$. Note that the family $e_{\in \Tc_{\beta+1}}$ is $s$-indiscernible over $A$. Since $d_E \notin \bddu_\lambda(A)$, we can find, by \cref{prop:explicit-bound}, a sequence $(\sigma_i)_{i<\omega}$ of elements of $\Aut(\Mb/A)$ such that for each $i<j<\omega$, $(\sigma_i \cdot d)\slashed{E}_{\beta+1}(\sigma_j\cdot d)$. Now chose $(e_f)_{f \in \Fc_{\beta+2}}$ in such a way that $e_{\in \Fc_{\beta+2}}$ extends what was already defined, is $s$-indiscernible over $A$, and is locally based on the family $(c_f)_{f \in \Fc_{\beta+2}}$ defined by $c_{\langle i \rangle \frown f}=\sigma_i \cdot e_f$ for all $f \in \Tc_{\beta+1}$ (which is possible by \cref{fact:modeling-for-s-ind}). In particular, note that for any $i<j<\omega$, we still have that $(e_{\zeta_{\beta+2}\frown \langle i \rangle},e_{\zeta_{\beta+2}\frown \langle j \rangle}) \equiv_A (d,\sigma_1\cdot d)$ and so, in particular, $e_{\zeta_{\beta+2}\frown \langle i \rangle}\slashed{E}_{\beta+1}e_{\zeta_{\beta+2}\frown \langle j \rangle}$.

  At limit stage $\beta$, given $e_{\in \Tc_{\beta}}$, note that this family is automatically $s$-indiscernible over $A$. Extend it to a family $e_{\in \Fc_{\beta+1}}$ that is $s$-indiscernible over $A$. (This is always possible.)
  
  \vspace{0.5em}  
  \noindent \emph{Claim.} For any $\beta < \delta < \mu^+$, if $E_{\beta+1} = E_{\delta+1}$, then $e_{\zeta_{\beta+1}}\slashed{E}_{\beta+1}e_{\zeta_{\delta+1}}$.
  \vspace{0.25em}

  \noindent \emph{Proof of claim.} The sequence $(e_{\zeta_{\beta+2}\frown \langle i \rangle})_{i<\omega}$ is $e_{\zeta_{\delta+1}}$-indiscernible. Since
  \[
    e_{\zeta_{\beta+2}\frown \langle 0 \rangle}\slashed{E}_{\beta+1}e_{\zeta_{\beta+2}\frown \langle 1 \rangle},
  \]
   it must be the case that $e_{\zeta_{\delta+1}}\slashed{E}_{\beta+1}e_{\zeta_{\beta+2}\frown \langle  i \rangle}$ for all $i<\omega$. \hfill $\square_{\text{claim}}$
   
   \vspace{0.5em}  
   Let $g$ be the partial function taking $\beta$ to $[e_{\zeta_{\beta+1}}]_{E_{\beta+1}}$. By the claim, this is an injection into $\bddu_\lambda(Ab_{\in \Tc_\alpha})\setminus \bddu_\lambda(A)$. By the choice of $\mu$, $g$'s domain cannot be cofinal in $\mu^+$, so the construction must have halted at some $\gamma < \mu^+$. 

   Extend $e_{\in \Tc_\gamma}$ to $e_{\in \Fc_{\gamma+1}}$ in such a way that the resulting family is $s$-indiscernible over $A$. Set $e_{\zeta_{\gamma+1}} = \varnothing$. Let $h$ be the function with domain $[\alpha,\gamma+1)$ such that $h(\delta) = 0$ for all $\delta \in [\alpha,\gamma)$ and $h(\gamma) = 1$.

   \vspace{0.5em}  
   \noindent \emph{Claim.} For any $d_E \in \bddu_\lambda(Ae_{\tgeq \zeta_\alpha}) \setminus \bddu_\lambda(A)$, $d_E \notin \bddu_\lambda(Ae_{\tgeq h})$.
   \vspace{0.25em}

   \noindent \emph{Proof of claim.} Assume that there is some $d_E \in  (\bddu_\lambda(Ae_{\tgeq \zeta_\alpha})\cap \bddu_\lambda(Ae_{\tgeq h})) \setminus \bddu_\lambda(A)$. By basic properties of $s$-indiscernibility, $e_{\in \Tc_\gamma}$ is $s$-indiscernible over $Ae_{\tgeq h}$.  By \cref{fact:modeling-for-s-ind}, we can find a tuple $e' \equiv_A e_{\in\Tc_\gamma}$ such that $e_{\in \Fc_{\gamma+1}}$ is $s$-indiscernible over $Ae'$. We can also find a model $M \supseteq Ae'$ with $|M|=\lambda$ such that $e_{\in \Fc_{\gamma+1}}$ is $s$-indiscernible over $M$. By \cref{cor:shrink-ultra}, there is an invariant equivalence relation $F$ such that $[M]_F$ and $d_E$ are interdefinable over $\varnothing$. This implies that $[M]_F \in \bddu_\lambda(Ae_{\tgeq \zeta_\alpha})\setminus \bddu_\lambda(A)$, but this contradicts the fact that the construction halted at step $\gamma$. \hfill $\square_{\text{claim}}$

   \vspace{0.5em}
   So, by the claim, we have that $\bddu_\lambda(Ae_{\tgeq \zeta_\alpha})\cap \bddu_\lambda(Ae_{\tgeq h}) = \bddu_\lambda(A)$. Therefore, by the choice of $\lambda$, $e_{\tgeq \zeta_\alpha}\indbu_A e_{\tgeq h}$, as required.
\end{proof}

\begin{thm}[Full existence]\label{thm:full-existence}
  For any set of hyperimaginaries $A$ and real tuples $b$ and $c$, there is $b' \equiv_A b$ such that $b' \indbu_A c$.
\end{thm}
\begin{proof}
  It is sufficient to show this in the special case that $b=c$. Specifically, given $d$ and $e$, if we can find $d'e'\equiv_A de$ such that $d'e'\indbu_A de$, then we have $d' \indbu_A e$ by monotonicity. So fix a set of hyperimaginaries $A$ and a real tuple $b$. We can now apply \cref{lem:full-existence-technical} to the family $(b_f)_{f \in \Tc_0}$ with $b_\varnothing = b$ to get a family $(e_f)_{f \in \Tc_{\gamma+1}}$ such that $e_{\zeta_0}=b$ for some $f \in \Tc_{\gamma+1}$, $b \equiv_A b_f$ and $b \indbu_A b_f$. Let $\sigma$ be an automorphism fixing $A$ taking $b_f$ to $b$. We then have that $b' = \sigma \cdot b$ is the required element, and we are done.
\end{proof}

\begin{cor}\label{cor:full-existence-with-ultra}
  For any set of hyperimaginaries $A$ and any ultraimaginaries $b_E$ and $c_F$, there is $b'_E \equiv_A b_E$ such that $b'_E \indbu_A c_F$.
\end{cor}
\begin{proof}
  Apply \cref{thm:full-existence} to $b$ and $c$ to get $b' \equiv_A b$ such that $b'\indbu_A c$. We then have that $\bddu(b') \indbu_A \bddu(c)$, so by monotonicity, $b'_E \indbu_A c_F$.
\end{proof}

\begin{cor}[Extension]\label{cor:extension}
  For any set of hyperimaginaries $A$ and any ultraimaginaries $b_E$, $c_F$, and $d_G$, if $b_E \indbu_A c_F$, then there is $b'_E \equiv_{Ac_F} b_E$ such that $b'_E \indbu_A c_Fd_G$. 
\end{cor}
\begin{proof}
 By \cref{cor:full-existence-with-ultra}, we can find $b'_E \equiv_{Ac_F} b$ such that $b'_E \indbu_{Ac_F} d_G$. By symmetry and transitivity, we have that $b'_E \indbu_{A}c_Fd_G$.   
\end{proof}

Compactness is very essential in the proof of \cref{fact:modeling-for-s-ind} and therefore also \cref{thm:full-existence}, which raises the following question.

\begin{quest}
  Does \cref{thm:full-existence} hold when $A$ is a set of ultraimaginaries?
\end{quest}

\section{Total \texorpdfstring{$\indbu$}{bu}-Morley sequences}

\begin{defn}
  A \emph{$\indbu$-Morley sequence over $A$} is an $A$-indiscernible sequence $(b_i)_{i<\omega}$ such that $b_i \indbu_A b_{<i}$ for each $i<\omega$.

  A \emph{weakly total $\indbu$-Morley sequence over $A$} is an $A$-indiscernible sequence $(b_i)_{i<\omega}$ such that for any finite $I$ and any $J$ (of any order type), if $I+J \equiv^{\EM}_{A} b_{<\omega}$, then $I \indbu_A J$.

  A \emph{total $\indbu$-Morley sequence over $A$} is an $A$-indiscernible sequence $(b_i)_{i<\omega}$ such that for any $I$ and $J$ (of any order type), if $I+J \equiv^{\EM}_Ab_{<\omega}$, then $I \indbu_A J$.
\end{defn}

We could write down stronger and weaker forms of the $\indbu$-Morley condition, but we are really only interested in total $\indbu$-Morley sequences, as they seem to be a fairly robust class (see \cref{thm:ind-bu-Morley-char}). Weakly total $\indbu$-Morley sequences seem to be the best we can get without large cardinals, however,  which does raise the following question.

\begin{quest}
  Is every weakly total $\indbu$-Morley sequence a total $\indbu$-Morley sequence? 
\end{quest}

One immediate property of total $\indbu$-Morley sequences is that they act as universal witnesses of Lascar strong type in a strong way. 

\begin{prop}\label{prop:Lstp-strong-witness}
  For any $A$ and $b$, if there is a total $\indbu$-Morley sequence $(b_i)_{i<\omega}$ over $A$ with $b_0 = b$, then for any $b'$, $b' \equiv^\Las_A b$ if and only if there are $I_0,J_0,I_1,\dots,J_{n-1},I_n$ such that $b \in I_0$,  $b' \in I_n$, and, for each $i<n$, $I_i + J_i$ and $I_{i+1} +J_i$ are both $A$-indiscernible and have the same EM-type as $b_{<\omega}$.
\end{prop}
\begin{proof}
  Let $I = (b_i)_{i<\omega}$. We only need to prove that if $b' \equiv^\Las_A b$, then the required configuration exists. Choose $I'$ so that $bI \equiv^\Las_A b'I'$. Then the required configuration exists by \cref{prop:crab-walk-char}. 
\end{proof}
A similar statement is true for weakly total $\indbu$-Morley sequences, which we will state in \cref{cor:Lascar-strong-type-clean-witness} after we have shown that weakly total $\indbu$-Morley sequences always exist without set-theoretic hypotheses.

\subsection{ Characterization of total \texorpdfstring{$\indbu$}{bu}-Morley sequences}

\begin{defn}\label{defn:approx-A}
  For any set of parameters $A$, we write $\approx_A$ for the transitive closure of the relation $I \sim_A J$ that holds if and only if $I$ and $J$ are infinite and either $I+J$ or $J+I$ is an $A$-indiscernible sequence.

  By an abuse of notation, we write $[I]_{\approx_A}$ for the ultraimaginary $[AI]_E$, where $E$ is the equivalence relation on tuples of the same length as $AI$ such that $E(AI,BJ)$ holds if and only if $A=B$ in our fixed enumeration and $I\approx_A J$.
\end{defn}

Note that we do not in general require that $I$ and $J$ have the same order type.

 We will also need an appropriate Lascar strong type generalization of Ehrenfeucht-Mostowski type. 

\begin{defn}\label{defn:LEM-type}
  Given two $A$-indiscernible sequences $I$ and $J$, we say that $I$ and $J$ have the same \emph{Lascar-Ehrenfeucht-Mostowski type} (or \emph{LEM-type}) \emph{over $A$}, written $I \equiv^{\LEM}_A J$, if there is some $J' \equiv^\Las_A J$ such that $I+J'$ is $A$-indiscernible.
\end{defn}

\begin{lem}\label{lem:approx-special-witness}
  For any infinite order types $O$ and $O'$, $I \approx_A J$ if and only if there are $K_0, L_0, K_1, \dots, L_{n-1}, K_n$ such that
  \begin{itemize}
  \item $K_0 = I$ and $K_n = J$,
  \item for $0 < i < n$, $K_i$ is a sequence of order type $O$,
  \item for $i<n$, $L_i$ is a sequence of order type $O'$, and
  \item for $i<n$, $K_i + L_i$ and $K_{i+1} + L_i$ are $A$-indiscernible.
  \end{itemize}
\end{lem}
\begin{proof}
  The $\Leftarrow$ direction is obvious.

  For the $\To$ direction, we will proceed by induction. First assume that $I \sim_A J$. If $I+J$ is $A$-indiscernible, then find $L$ of order type $O'$ such that $I+J+L$ is $A$-indiscernible. We then have that $I+L$ and $J+L$ are $A$-indiscernible. If $J+I$ is $A$-indiscernible, then find $L$ of order type $O'$ such that $J+I+L$ is $A$-indiscernible. We then have that $I+L$ and $J+L$ are $A$-indiscernible.

  Now assume that we know the statement holds for any $I$ and $J$ such that there is a sequence $I'_0,\dots,I'_n$ with $I'_0 = I$, $I'_n = J$, and $I'_i \sim_A I'_{i+1}$ for each $i<n$. Now assume that there is a sequence $I'_0, \dots, I'_{n+1}$ with $I'_0 = I$, $I'_{n+1} = J$, and $I'_i \sim_A I'_{i+1}$ for each $i\leq n$. Apply the induction hypothesis to get $K_0,L_0,K_1,\dots,L_{m-1},K_m$ satisfying the properties in the statement of the lemma with $K_0 = I$ and $K_m = I'_n$. Now since $I'_n \sim_A I'_{n+1} = J$, we can apply the $n=1$ case to get $L_m$ such that $I'_n + L_m$ and $I'_{n+1} + L_m$ are both $A$-indiscernible. By compactness, we can find $K_m$ of order type $O$ such that $K_m +L_m$ and $K_m+L_{m-1}$ are both $A$-indiscernible. We then have that $K_0,L_0,K_1,\dots,K_m,L_m,K_{m+1}$ is the require sequence, where $K_{m+1}=J$.
  \end{proof}

\begin{prop}\label{prop:LEM-basic-properties}
  Fix a set of hyperimaginary parameters $A$.
  \begin{enumerate}[(1)]
  \item $\equiv^{\LEM}_A$ is an equivalence relation on the class of $A$-indiscernible sequences.
  \item If $I \equiv^\Las_A J$, then $I \equiv^{\LEM}_A J$.
  \item If $I$ and $J$ have the same order type, then $I \equiv^\Las_A J$ if and only if $I \equiv^{\LEM}_A J$.
  \item If $I \equiv^{\LEM}_A J$, then $I \equiv^{\EM}_A J$.
  \item If $I \approx_A J$, then $I \equiv^{\LEM}_A J$. 
  \end{enumerate}
\end{prop}
\begin{proof}
  Recall the following fact: If $I$ and $J$ have the same order type and $I+J$ is $A$-indiscernible, then $I\equiv^\Las_A J$.
  
  For (1), it follows easily from the fact that $\equiv^{\LEM}_A$ is reflexive. For symmetry, assume that $I \equiv^{\LEM}_A J$, and let $I$, $J$, and $J'$ be as in the definition of $\equiv^{\LEM}$. Find $I'$ such that $IJ' \equiv^\Las_A I'J$. Then extend $I'+J$ to $I'+J+I''$, where $I''$ has the same order type as $I$. We then have that $I'' \equiv^\Las_A I' \equiv^\Las_A I$, so $J \equiv^{\LEM}_A I$. Finally, assume that $I \equiv^{\LEM}_A J$ and $J \equiv^{\LEM}_A K$. Let this be witnessed by $J'$ and $K'$ such that $I+J'$ and $J+ K'$ are $A$-indiscernible.  Find $K''$ with the same order type as $K$ such that $I+J'+K''$ is $A$-indiscernible. Then find $K^\ast$ such that $J'K'' \equiv^\Las_A JK^\ast$. We have that $K \equiv^\Las_A K^\ast \equiv^\Las_A K''$, so $I \equiv^{\LEM}_A K$.

  (2) and (3) are immediate from the fact. (4) is obvious. 
  
  For (5), let $I$ and $J$ be such that $I \approx_A J$. Let $K_0,L_0,K_1,\dots,L_{n-1},K_{n}$ be as in \cref{lem:approx-special-witness} where the $K_i$'s and $L_i$'s (other than $K_0$) have the same order type as $J$. We have that $K_i \equiv^\Las_A L_i$ and $K_i \equiv^\Las_A L_{i-1}$ for each $i > 0$. Therefore we can take $L_0$ to be the required $J'$. 
\end{proof}

Now we will see that total $\indbu$-Morley sequences over $A$ are precisely those which are `as generic as possible' in terms of $\approx_A$ (i.e., their $\equiv^{\LEM}_A$-equivalence class decomposes into a single $\approx_A$-equivalence class).

\begin{thm}\label{thm:ind-bu-Morley-char}
  For any $A$-indiscernible sequence $(b_i)_{i<\omega}$ (with $A$ a set of hyperimaginary parameters), the following are equivalent.
  \begin{enumerate}[(1)]
  \item $b_{<\omega}$ is a total $\indbu$-Morley sequence over $A$.
  \item There exists a pair of infinite sequences $I$ and $J$ (of any, possibly distinct order types) such that $I+J \equiv^{\EM}_A b_{<\omega}$ and $I \indbu_A J$. 
  \item For any $I$, $I \approx_A b_{<\omega}$ if and only if $I \equiv^{\LEM}_A b_{<\omega}$.
  \item $[b_{<\omega}]_{\approx_A} \in \bddu(A)$. 
  \end{enumerate}
\end{thm}
\begin{proof}
  (1)$\To$(2). This is immediate from the definition.
\vspace{0.5em}
  
\noindent  (2)$\To$(3). Let $I$ and $J$ be as in the statement of (2). By \cref{prop:LEM-basic-properties}, we obviously have that if $K \approx_A I$, then $K \equiv^{\LEM}_A I$. Conversely, assume that $K \equiv^{\LEM}_A I$, so there is $I'\equiv^\Las_A I$ such that $K+I'$ is $A$-indiscernible. Since $I \indbu_A J$, \cref{prop:ind-bu-char} allows us to find $I_0, J_0, I_1,\dots, J_{n-1}, I_n$ such that $I_0 = I$, $J_0 = J$, $I_n = I'$, and, for each $i<n$, $I_i \equiv^\Las_{AJ_i} I_{i+1}$ and $J_i \equiv^\Las_{AI_{i+1}}J_{i+1}$, if $i<n-1$. In particular, this implies that for each $i < n$, $I_i + J_i$ and $I_{i+1} + J_i$ are both $A$-indiscernible, so $I \approx_A I' \sim_A K$, whence $I \approx_A K$.
\vspace{0.5em}

\noindent (3)$\To$(1). Since $I\approx_A J$ always implies $I \equiv^{\LEM}_A J$, we only need to prove one direction.

Assume (3), and let $I$ and $J$ be such that $I+J \equiv^{\EM}_{A} b_{<\omega}$. Let $I'$ be some sequence of the same order type as $I$. Suppose furthermore that $I \equiv^\Las_A I'$. This implies that $I \equiv^{\LEM}_A I'$, so, by assumption, we have that $I \approx_A I'$. By \cref{lem:approx-special-witness}, we can find $I_0,J_0,I_1,\dots,J_{n-1},I_n$ such that $I_0 = I$, $J_0 = J$, $I_n = I'$, each $I_i$ has the same order type as $I$, each $J_i$ has the same order type as $J$, and $I_i + J_i$ and $I_{i+1}+J_i$ are both $A$-indiscernible for each $i<n$.

It is generally true that if $K+L$ and $K+L'$ are both $A$-indiscernible, then $L \equiv^\Las_{AK} L'$. Therefore, we have that for each $i<n$, $I_i \equiv^\Las_{AJ_i} I_{i+1}$ and, for each $i<n-1$, $J_{i}\equiv^\Las_{AI_{i+1}} J_{i+1}$. Since we can do this for any $I' \equiv^\Las_A I$, \cref{prop:ind-bu-char} yields that $I \indbu_A J$.
\vspace{0.5em}

\noindent (3)$\To$(4). This follows immediately from \cref{prop:LEM-basic-properties} and the fact that $[b_{<\omega}]_{\equiv^\Las_A}$ is always in $\bddu(A)$.
\vspace{0.5em}

\noindent (4)$\To$(3). Let $I \equiv^{\LEM}_A b_{<\omega}$. Find $I'$ such that $I \equiv^\Las_A I'$ and $b_{<\omega} + I'$ is $A$-indiscernible. Since $[b_{<\omega}]_{\approx_A} \in \bddu(A)$, we must also have that $[I']_{\approx_A} \in \bddu(A)$. By \cref{prop:Lstp-char}, there must be an automorphism $\sigma \in \Aut(\Mb/A,[I']_{\approx A})$ such that $\sigma \cdot I' = I$. Therefore $I\approx_A I'$ and hence $I \approx_A b_{<\omega}$.
\end{proof}

\subsection{Building ((weakly) total) \texorpdfstring{$\indbu$}{bu}-Morley sequences}

\noindent Given that $\indbu$ satisfies full existence, an immediate, familiar Erd\"os-Rado argument gives  that $\indbu$-Morley sequences exist, but in the end we will need a technical strengthening of this result.

  \begin{prop}\label{prop:Mor-seq-exist-technical}
    If $(b_f)_{f \in \Tc_\alpha}$ is $s$-indiscernible over $A$, then there is a family $(c_f)_{f\in\Fc_{\alpha+1}}$ such that
    \begin{itemize}
    \item $c_{\in \Fc_{\alpha+1}}$ is $s$-indiscernible over $A$,
    \item $c_{\iota_{\alpha,\alpha+1}(f)} = b_f$ for each $f \in \Tc_\alpha$, and
    \item the sequence $(c_{\tgeq \langle i \rangle})_{i < \omega}$ is an $\indbu$-Morley sequence over $A$. 
    \end{itemize}
  \end{prop}
  \begin{proof}
    Let $\kappa$ be sufficiently large to apply Erd\"os-Rado to a sequence of tuples of the same length as $b_{\in \Tc_\alpha}$ over the set $A$.

    Let $\gamma(0) = \alpha$. Let $c^0_f = b_f$ for all $f \in \Tc_{\gamma(0)} = \Tc_\alpha$. Let $g_0 = \varnothing$ (as an element of $\Tc_\alpha$).

    At successor stage $\beta+1$, assume we have $(c^\beta_f)_{\Tc_{\gamma(\beta)}}$ which is $s$-indiscernible over $A$ and which satisfies $c^\beta_{\iota_{\delta,\beta}(f)} = c^\delta_f$ for all $\delta \in [\alpha,\beta)$. By \cref{lem:full-existence-technical}, we can build a family $(c^{\beta+1}_f)_{\Tc_{\gamma(\beta+1)}}$ (for some successor ordinal $\gamma(\beta+1) > \gamma(\beta)$) such that
    \begin{itemize}
    \item $(b^{\beta+1}_f)_{f \in \Tc_{\gamma(\beta+1)}}$ is $s$-indiscernible over $A$,
    \item for each $f \in \Tc_{\gamma(\beta)}$, $c^{\beta}_{f} = c^{\beta+1}_{\iota_{\gamma(\beta),\gamma(\beta+1)}(f)}$, and
    \item $c^{\beta+1}_{\tgeq \zeta_(\gamma(\beta))} \indbu_A c^{\beta+1}_{\tgeq h}$, where $\dom(h) = [\gamma(\beta),\gamma(\beta+1))$, $h(\delta) = 0$ for all $\delta \in [\gamma(\beta),\gamma(\beta+1)-1)$, and $h(\gamma(\beta+1)-1) = 1$. 
    \end{itemize}
    Let $g_{\beta+1} \in \Tc_{\gamma(\beta+1)}$ be the function with domain $[\alpha,\gamma(\beta+1))$ which has $h(\delta) = 0$ for all $\delta \in [\alpha,\gamma(\beta+1)-1)$ and $h(\gamma(\beta+1)-1) =1$. Note that $g_{\beta+1}\tgeq h$. Also note that by induction we have that
    \[
\textstyle      c^{\beta+1}_{\tgeq g_{\beta+1}} \indbu_A \{c^{\beta+1}_{\tgeq \iota_{\gamma(\delta),\gamma(\beta+1)}(g_{\delta})} : \delta < \beta+1\},
\]
since $\iota_{\gamma(\delta),\gamma(\beta+1)}(g_\delta)\tgeq \zeta_{\gamma(\beta)}$ for all $\delta < \beta+1$.
    
At limit stage $\beta$, let $\gamma(\beta) = \sup_{\delta<\beta}\gamma(\delta)$ and let $(c^\beta_f)_{f \in \Tc_{\gamma(\beta)}}$ be the direct limit of $(c^{\delta}_f)_{f \in\Tc_{\gamma(\delta)}}$ for $\delta < \beta$. Leave $g_\beta$ undefined.

Stop once we have $(b^\kappa_{f})_{f \in \Tc_{\gamma(\kappa)}}$. Consider the sequence $(b^\kappa_{\tgeq \iota_{\gamma(\beta),\kappa}(g_\beta)})_{\beta \in \kappa \setminus \lim \kappa}$.\footnote{We write $\lim \alpha$ for the set of limit ordinals in $\alpha$.} By our choice of $\kappa$ and a standard application of the Erd\"os-Rado theorem, we can find a family $(c_f)_{f\in \Fc_{\alpha+1}}$ such that the sequence $(c_{\tgeq \langle i \rangle})_{i<\omega}$ is $A$-indiscernible and for every increasing tuple $\bar\imath < \omega$, there is $\bar\beta \in \kappa\setminus \lim\kappa$ such that $c_{\tgeq\langle  i_0\rangle}\dots c_{\tgeq \langle  i_{k} \rangle} \equiv_A b^\kappa_{\tgeq \iota_{\gamma(\beta_0),\kappa}(g_{\beta_0})}\dots b^\kappa_{\tgeq \iota_{\gamma(\beta_{k}),\kappa}(g_{\beta_k})}$.

In particular, note that this implies that
\[
\textstyle  c_{\tgeq \langle i \rangle} \indbu_A \{c_{\tgeq \langle j \rangle} : j < i\}
\]
for every $i<\omega$. Clearly by applying an automorphism, we may assume that $c_{\iota_{\alpha,\alpha+1}}(f) = b_f$ for each $f \in \Tc_\alpha$, so all we need to do is show that the family $c_{\in \Fc_{\alpha+1}}$ is $s$-indiscernible over $A$.

Since the sequence $(c_{\tgeq \langle i \rangle})_{i < \omega}$ is $A$-indiscernible, it is sufficient, by induction, to show the following statement: For any sequence $\bar{f}_0,\bar{f}_1,\dots,\bar{f}_k,\dots,\bar{f}_\ell$ of tuples of elements of $\Fc_{\alpha+1}$ satisfying $\bar{f}_i \tgeq \langle i \rangle$ for all $i\leq \ell$ and any $\bar{h} \tgeq \langle k \rangle$ such that $\bar{f}_k$ and $\bar{h}$ realize the same quantifier-free type, we have that $c_{\bar{f}_k}$ and $c_{\bar{h}}$ realize the same type over $Ac_{\bar{f}_0}\dots c_{\bar{f}_{k-1}}c_{\bar{f}_{k+1}}\dots c_{\bar{f}_\ell}$.

So let $\bar{f}_0,\dots,\bar{f}_\ell$ and $\bar{h}$ be as in the statement. By construction, there are $\beta_0,\dots,\beta_{\ell}$ such that $c_{\tgeq \langle i \rangle} \equiv_A b^\kappa_{\tgeq \iota_{\gamma(\beta_i),\kappa}(g_{\beta_i})}$ for each $i \leq \ell$. Let $\bar{f}'_0,\dots,\bar{f}'_\ell,\bar{h}'$ be the corresponding elements of $\Tc_\kappa$. (So, in particular, $\bar{f}'_i \tgeq g_{\beta_i}$ for each $i \leq \ell$ and $\bar{h}' \tgeq g_{\beta_k}$). We now have that $\bar{f}'_k$ and $\bar{h}'$ realize the same quantifier-free type. Therefore, by the $s$-indiscernible of $b^\kappa_{\in \Tc_\kappa}$, we have that $b^\kappa_{\bar{f}'_k}$ and $b^\kappa_{\bar{h}'}$ realize the same type over $Ab^\kappa_{\bar{f}'_0}\dots b^\kappa_{\bar{f}'_{k-1}}b^\kappa_{\bar{f}'_{k+1}}\dots b^\kappa_{\bar{f}'_\ell}$. From this the required statement follows, and we have that $c_{\in \Fc_{\alpha+1}}$ is $s$-indiscernible over $A$.
\end{proof}

\begin{cor}\label{cor:Mor-seq-exists}
  For any set of hyperimaginaries $A$ and any real tuple $b$, there is an $\indbu$-Morley sequence $(b_i)_{i<\omega}$ over $A$ with $b_0 = b$. 
\end{cor}
\begin{proof}
  Apply \cref{prop:Mor-seq-exist-technical} to the tree $(b_f)_{f \in \Tc_{0}}$ defined by $b_\varnothing = b$. 
\end{proof}

The order type $\omega$ is essential, however; Erd\"os-Rado only guarantees the existence of sequences that satisfy the relevant condition on finite tuples.  Fortunately, this is more than sufficient for the following weak `chain condition.'

\begin{lem}\label{lem:weak-chain-condition}
  If $(b_i)_{i<\omega}$ is an $\indbu$-Morley sequence over $A$ that is moreover $Ac$-indiscernible, then $b_0 \indbu_A c$.
\end{lem}
\begin{proof}
  Fix $\lambda$. Let $\mu = |\bddu_\lambda(Ac)\setminus \bddu_\lambda(A)|$. Extend $b_{<\omega}$ to $(b_i)_{i<\mu^+}$. We still have that for any $i<j<\mu^+$, $b_i \indbu_A b_j$ (since this is only a property of $\tp(b_ib_j/A)$). Therefore the sets $\bddu_\lambda(Ab_i)\setminus \bddu_\lambda(A)$ are pairwise disjoint. Since there are $\mu^+$ many of them, one of them must be disjoint from $\bddu_\lambda(Ac)\setminus \bddu_\lambda(A)$. Therefore by indiscernibility, we must have $b_0 \indbu_A c$.  
\end{proof}

We will not use the following corollary of \cref{lem:weak-chain-condition}, but it is worth pointing out.

  \begin{cor}
      If $I$ is a total $\indbu$-Morley sequence over $A$ that is $Ac$-indiscernible, then $I \indbu_A c$.
\end{cor}
\begin{proof}
  Extend $I$ to an $Ac$-indiscernible sequence $I_0 + I_1 + I_2 + \dots$ with $I_0 = I$. Since $I$ is totally $\indbu$-Morley, we have that $(I_i)_{i<\omega}$ is an $\indbu$-Morley sequence over $A$. So by \cref{lem:weak-chain-condition}, we have $I = I_0 \indbu_A c$.
\end{proof}

Part (2) of following definition is equivalent to \cite[Def.~2.1, 3.4]{KKS2013} in our context. The rest of it is based on \cite[Def.~5.7]{KaplanRamseyOnKim}.

  \begin{defn}\label{defn:spread-out}
  Fix a family $(b_f)_{f \in \Tc_\alpha}$.
  \begin{enumerate}[(1)]
  \item For $w \subseteq \alpha \setminus \lim \alpha$, the \emph{restriction of $\Tc_\alpha$ to the set of levels $w$} is given by
    \[
      \Tc_\alpha \res w = \{f \in \Tc_\alpha : \min \dom (f) \in w,~\beta \in \dom(f)\setminus w \To f(\beta)=0\}.
    \]
  \item     A family $(b_f)_{f \in \Tc_\alpha}$ is \emph{$str$-indiscernible over $A$} if it is $s$-indiscernible over $A$ and satisfies that for any $w,v \in [\alpha \setminus \lim \alpha]^{<\omega}$ with $|w|=|v|$, $b_{\in \Tc_\alpha\res w}$ and $b_{\in \Tc_\alpha\res v}$ realize the same type over $A$. 
  \item     We say that $b_{\in \Tc_\alpha}$ is \emph{$\indbu$-spread out over $A$} if for any $f \in \Tc_\alpha$ (with $\dom(f) = [\beta+1,\alpha)$ for some $\beta < \alpha$), the sequence $(b_{\tgeq f \frown \langle i \rangle})_{i < \omega}$ is an $\indbu$-Morley sequence over $A$.
  \item $b_{\in \Tc_\alpha}$ is an \emph{$\indbu$-Morley tree over $A$} if it is $\indbu$-spread out and $str$-indiscernible over $A$.
  \end{enumerate}
\end{defn}

Note that if $b_{\in \Tc_\alpha}$ is $\indbu$-spread out over $A$, then any restriction $b_{\in \Tc_\alpha \res w}$ is also $\indbu$-spread out over $A$ (even for infinite $w$). Also note that, by a basic compactness argument, if $\alpha$ is infinite and $(b_f)_{f \in \Tc_\alpha}$ is $str$-indiscernible over $A$, then for any $\beta$, we can find a tree $(c_f)_{f \in \Tc_\beta}$ which is $str$-indiscernible over $A$ such that for any $w \in [\alpha]^{<\omega}$ and $v \in [\beta]^{<\omega}$ with $|w|=|v|$, $b_{\in \Tc_\alpha \res w} \equiv_A c_{\in \Tc_\beta \res v}$.

\begin{prop}\label{prop:weakly-ind-bu-spread-out-trees-exist}
  For any $A$, $b$, and $\kappa$, there is a tree $(b_f)_{f \in \Tc_\kappa}$ that is $\indbu$-spread out and $s$-indiscernible over $A$ such that for each $f \in \Tc_\kappa$, $b_f \equiv_A b$. 
\end{prop}
\begin{proof}
  Let $(b^0_f)_{f \in \Tc_0}$ be defined by $b^0_\varnothing = b$. This is vacuously $\indbu$-spread out and $s$-indiscernible out over $A$.  

  At successor stage $\alpha+1$, given $(b^\alpha_f)_{f \in \Tc_\alpha}$ which is $\indbu$-spread out and $s$-indiscernible by \cref{prop:Mor-seq-exist-technical}, we can find an extension $(b^{\alpha+1}_f)_{f \in \Fc_{\alpha+1}}$ satisfying $b^{\alpha+1}_{\iota_{\alpha,\alpha+1}(f)} = b^\alpha_f$ for all $f \in \Tc_\alpha$ such that $b^{\alpha+1}_{ \in \Fc_{\alpha+1}}$ is $s$-indiscernible over $A$  and $(b^{\alpha+1}_{\tgeq \langle i \rangle})_{i < \omega}$ is an $\indbu$-Morley sequence over $A$. By \cref{fact:modeling-for-s-ind}, we can find $b^{\alpha+1}_\varnothing \equiv_A b$ such that the tree $(b^{\alpha+1}_f)_{f \in \Tc_{\alpha+1}}$ is $s$-indiscernible over $A$. By construction, we now have that $(b^{\alpha+1}_f)_{f \in \Tc_{\alpha+1}}$ is $\indbu$-spread out over $A$.

  At limit stage $\alpha$, let $(b^\alpha_f)_{f \in \Tc_\alpha}$ be the direct limit of $(b^{\beta}_{f})_{f \in \Tc_\beta}$ for $\beta < \alpha$. It is immediate from the definitions that $b^\alpha_{\in \Tc_\alpha}$ is $\indbu$-spread out and $s$-indiscernible over $A$.

  Once we have constructed $(b^\kappa_f)_{f \in \Tc_\kappa}$, let $b_f = b^\kappa_f$ for each $f \in \Tc_\kappa$. We have that $b_{ \in \Tc_\kappa}$ is the required tree by induction.
  \end{proof}

By the same argument as in \cite[Lem.~5.10]{KaplanRamseyOnKim}, we get the following.

\begin{lem}\label{lem:Extract-weakly-ind-bu-Morley-tree-with-Erdos-Rado}
  Suppose $(b_f)_{f \in \Tc_\kappa}$ is $\indbu$-spread out and $s$-indiscernible over $A$ with all $b_f$ tuples of the same length. If $\kappa$ is sufficiently large, then there is an $\indbu$-Morley tree $(c_f)_{f \in \Tc_\omega}$ such that for any $w \in [\omega]^{<\omega}$, there is $v \in [\kappa]^{<\omega}$ such that
  \[
    (b_f)_{f \in \Tc_\kappa\res v} \equiv_A (c_f)_{f \in \Tc_\omega\res w}. 
  \]
\end{lem}

\begin{proof}
  For any tree $(b_f)_{f \in \Tc_\kappa}$ where all tuples $b_f$ are the same length, define a function $t$ on $[\kappa\setminus \lim \kappa]^{<\omega}$ that takes $w \in [\kappa\setminus \lim\kappa]^{<\omega}$ to $\tp(b_{\in \Tc_\kappa \res w}/A)$. Let $\kappa$ be large enough to apply Erd\"os-Rado with the appropriate number of colors.

  By \cref{prop:weakly-ind-bu-spread-out-trees-exist}, we can find a tree $(b_f)_{f \in \Tc_\kappa}$ that is $s$-indiscernible and $\indbu$-spread out over $A$. By the typical Erd\"os-Rado argument, we can find a family $F$ of finite sets $w \subset \kappa\setminus \lim \kappa$ closed under subsets such that for any $w,v \in F$ with $|w|=|v|$, $b_{\in\Tc_\kappa \res w}\equiv_A b_{\in \Tc_\kappa\res v}$. By compactness, we can build a tree $(c_f)_{f \in \Tc_\omega}$ such that for any $w \in [\omega]^{<\omega}$, if $v \in F$ has $|w| = |v|$, then $c_{\in \Tc_\omega \res w} \equiv_A b_{\in \Tc_\kappa \res v}$. Finally, we have that $c_{\in \Tc_\omega\res [0,n)}$ is an $\indbu$-Morley tree for every $n < \omega$ because $b_{\in \Tc_\kappa}$ was $s$-indiscernible and $\indbu$-spread out over $A$. Therefore $c_{\in \Tc_\omega}$ is $s$-indiscernible and $\indbu$-spread out over $A$, and so $c_{\in \Tc_\omega}$ is an $\indbu$-Morley tree over $A$.
\end{proof}

\begin{prop}\label{prop:weak-ind-bu-Morley-tree-implies-weakly-total-ind-bu-Morley-seq}
  If $(b_f)_{f \in \Tc_\omega}$ is an $\indbu$-Morley tree over $A$, then $(b_{\zeta_\beta})_{\beta<\omega}$ is a weakly total $\indbu$-Morley sequence over $A$.
\end{prop}
\begin{proof}
  Fix a linear order $O$. Let $c_\alpha = b_{\zeta_\alpha}$ for each $\alpha < \omega$.

  For each positive $n<\omega$ and each $i<j<\omega$, we have that $b_{\tgeq \zeta_{n}\frown \langle i \rangle} \indbu_A b_{\tgeq \zeta_{n}\frown \langle j \rangle}$ and that the sequence $(b_{\tgeq \zeta_{n}\frown \langle i\rangle})_{i < \omega}$ is $Ac_{\geq n}$-indiscernible. By compactness, we can find $(c_i)_{i\in O}$ such that $(c_i)_{i \in \omega + O}$ is $A$-indiscernible and such that $(b_{\tgeq \zeta_{n}\frown \langle i\rangle})_{i < \omega}$ is $Ac_{\in [n,\omega)+O}$-indiscernible for each $n<\omega$.

  Therefore, by \cref{lem:weak-chain-condition}, we have that $c_{<n} \indbu_A c_{\in [n,\omega)+O}$. Hence, $(b_{\zeta_\beta})_{\beta < \omega}$ is a weakly total $\indbu$-Morley sequence.
\end{proof}

\begin{cor}\label{cor:Lascar-strong-type-clean-witness}
  For any $A$ and $b$, there is an $A$-indiscernible sequence $(b_i)_{i<\omega}$ with $b_0 = b$ such that for any $b' \equiv^\Las_A b$ and $n<\omega$, there are $I_0,J_0,I_1,J_1\dots,J_{k-1},I_k$ with
  \begin{itemize}
  \item $b$ the first element of $I_0$,
  \item  $b'$ the first element of $I_k$,
  \item  $|I_i| = n$ for all $i\leq k$,
  \item  $J_i$ infinite for all $i<k$, and
  \item  $I_i + J_i$ and $I_{i+1}+J_i$ realizing the same EM-type over $A$ as $b_{<\omega}$ for all $i<k$.
  \end{itemize}
  We can also arrange it so that $I_i$ is infinite for all $i \leq k$, $|J_i| = n$ for all $i<k$, and $I_i+J_i$ and $I_{i+1}+J_i$ realize the same EM-type over $A$ as $b_{<\omega}$ in the reverse order for all $i<k$ (with the same choice of $b_{<\omega}$ but possibly a different $k$).
\end{cor}
\begin{proof}
  This follows immediately from the same reasoning as in \cref{prop:Lstp-strong-witness} together with \cref{lem:Extract-weakly-ind-bu-Morley-tree-with-Erdos-Rado} and \cref{prop:weak-ind-bu-Morley-tree-implies-weakly-total-ind-bu-Morley-seq}, which establish that weakly total $\indbu$-Morley sequences always exist.
\end{proof}

To go further, we will need the following fact from \cite{silver_1966}. Recall that the statement $\kappa \to (\alpha)^{<\omega}_\gamma$ means that whenever $f: [\kappa]^{<\omega} \to \gamma$ is a function, there is a set $X \subseteq \kappa$ of order type $\alpha$ such that for each $n<\omega$, $f$ is constant on $[X]^n$.

\begin{fact}[Silver \cite{silver_1966}]\label{fact:Silver-fact}
 For any limit ordinal $\alpha$, if $\kappa$ is the smallest cardinal satisfying $\kappa \to (\alpha)^{<\omega}_2$, then for any $\gamma < \kappa$, $\kappa \to (\alpha)^{<\omega}_\gamma$. Furthermore, $\kappa$ is inaccessible.
\end{fact}

The smallest cardinal $\lambda$ satisfying $\lambda \to (\alpha)^{<\omega}_2$ is called the Erd\"os cardinal $\kappa(\alpha)$.  In the specific case of $\alpha = \omega$, we will also need the following fact. (See \cite[Exercise~7.4.9]{drake1974set} for a proof that easily generalizes to a proof of this fact.)  
  \begin{fact}\label{fact:bump-up}    
    If  $\kappa \to (\omega)^{<\omega}_\gamma$, then $(\gamma^{<\kappa})^+ \to (\omega+1)^{<\omega}_\gamma$.  
  \end{fact}

  In particular, if $\kappa(\omega)$ exists, then for any $\gamma < \kappa(\omega)$, $\kappa(\omega)^+ \to (\omega+1)^{<\omega}_\gamma$.

\begin{lem}\label{lem:Erdos-card-extract-tree}
  Suppose $(b_f)_{f \in \Tc_\lambda}$ is $\indbu$-spread out and $s$-indiscernible over $A$ with all $b_f$ tuples of the same length. If $\lambda\to (\omega + 1)^{<\omega}_{2^{|Ab|+|T|}}$, then there is a set $X \subseteq \lambda\setminus \lim \lambda$ with order type $\omega+1$ such that $b_{\in \Tc_\lambda\res w}$ is an $\indbu$-Morley tree over $A$.
\end{lem}
\begin{proof}
  Let the function $t$ be as in the proof of \cref{lem:Extract-weakly-ind-bu-Morley-tree-with-Erdos-Rado}. By assumption, we can find $X \subset \lambda \setminus \lim \lambda$ of order type $\omega+1$ such that $t$ is homogeneous on $X$. The tree $b_{\in \Tc_{\lambda}\res X}$ is clearly $s$-indiscernible over $A$, so the only thing to check is that it is $\indbu$-spread out over $A$, but this follows from the general fact that restrictions of trees $\indbu$-spread out over $A$ are $\indbu$-spread out over $A$.
\end{proof}

\begin{thm}\label{thm:Erdos-total-ind-bu-Morley-seq}
  For any $A$ and $b$ in any theory $T$, if there is a cardinal $\lambda$ satisfying $\lambda\to (\omega + 1)^{<\omega}_{2^{|Ab|+|T|}}$, then there is a total $\indbu$-Morley sequence $(b_i)_{i<\omega}$ over $A$ with $b_0 = b$.

  In particular, it is enough if there is an Erd\"os cardinal $\kappa(\alpha)$ such that $|Ab|+|T| < \kappa(\alpha)$ (for any limit $\alpha \geq \omega$). 
\end{thm}
\begin{proof}
If the Erd\"os cardinal $\kappa(\alpha)$ exists and $|Ab|+|T| < \kappa(\alpha)$, then by \cref{fact:Silver-fact}, we have $2^{|Ab|+|T|}<\kappa(\alpha)$ as well. Then if $\alpha= \omega$, we have that $\kappa(\alpha)^+ \to (\omega+1)^{<\omega}_{2^{|Ab|+|T|}}$ by \cref{fact:bump-up}. If $\alpha > \omega$, we clearly have $\kappa(\alpha) \to (\omega+1)^{<\omega}_{2^{|Ab|+|T|}}$ by \cref{fact:Silver-fact}. So in any such case we have the required $\lambda$.

  Let $\lambda$ be a cardinal such that $\lambda \to (\omega+1)^{<\omega}_{2^{|Ab|+|T|}}$ holds. By \cref{prop:weakly-ind-bu-spread-out-trees-exist}, we can build a tree $(b_f)_{f \in \Tc_\lambda}$ that is $s$-indiscernible and $\indbu$-spread out over $A$. By \cref{lem:Erdos-card-extract-tree} and the choice of $\lambda$, we can extract an $\indbu$-Morley tree $(c_f)_{f \in \Tc_{\omega+1}}$ from this. 

  By compactness, we can extend this to a tree $(c_f)_{f \in \Tc_{\omega+\omega}}$ that is $str$-indiscernible over $A$. We still have that for any $i<j<\omega$,
  \[
 \textstyle   c_{\tgeq \zeta_{\omega+1}\frown \langle  i \rangle} \indbu_A c_{\tgeq \zeta_{\omega+1}\frown \langle j \rangle}
  \]
  but now we also have that the $(c_{\tgeq \zeta_{\omega+1}\frown \langle  i \rangle})_{i<\omega}$ is $A\cup\{c_{\zeta_{\omega+i}} : i < \omega\}$-indiscernible, by $str$-indiscernibility of the full tree $c_{\in \Tc_{\omega+\omega}}$. Therefore, by \cref{lem:weak-chain-condition},
  \[
\textstyle    c_{\tgeq \zeta_{\omega+1}\frown \langle 0 \rangle} \indbu_A \{c_{\zeta_{\omega+i}} : i < \omega\},
  \]
  so in particular,
  \[
  \textstyle  \{c_{\zeta_{i}} : i < \omega\} \indbu_A \{c_{\zeta_{\omega+i}} : i < \omega\}.
  \]
  Let $d_i = c_{\zeta_{\alpha(i)}}$ for each $i<\omega+\omega$. We have that $(d_i)_{i< \omega + \omega}$ is $A$-indiscernible. Furthermore, by \cref{thm:ind-bu-Morley-char}, we have that $d_{<\omega}$ is a total $\indbu$-Morley sequence. By applying an automorphism, we get the required $b_{<\omega}$.
\end{proof}

So if we assume that for every $\lambda$, there is a $\kappa$ such that $\kappa \to (\omega+1)^{<\omega}_\lambda$, we get that Lascar strong type is always witnessed by total $\indbu$-Morley sequences in the manner of \cref{prop:Lstp-strong-witness}.

The use of large cardinals in \cref{thm:Erdos-total-ind-bu-Morley-seq} leaves an obvious question.

\begin{quest}
  Does the statement `for every $A$ and $b$, there is a total $\indbu$-Morley sequence $(b_i)_{i<\omega}$ over $A$ with $b_0 = b$' have any set-theoretic strength? What if we add cardinality restrictions, such as $|A|+|T|\leq \aleph_0$ and $|b| < \aleph_0$?
\end{quest}

\subsection{Total \texorpdfstring{$\indbu$}{bu}-Morley sequences in tame theories}

\noindent \cref{lem:weak-chain-condition} can be used to show that $\indd $ implies $\indbu$ (where $b\indd_A c$ means that $\tp(b/Ac)$ does not divide over $A$), something which was previously established for bounded hyperimaginary independence, $\indb$, in \cite{conant2021separation} and which is folklore for algebraic independence, $\inda$.

\begin{prop}
  For any real elements $A$, $b$, and $c$, if $b \indd_A c$, then $b \indbu_A c$.
\end{prop}
\begin{proof}
  Let $(c_i)_{i<\omega}$ be an $\indbu$-Morley sequence over $A$ with $c_0 = c$. Since $b \indd_{A}c$, we may assume that $c_{<\omega}$ is $Ab$-indiscernible. Therefore, by \cref{lem:weak-chain-condition}, $b \indbu_A c$.
\end{proof}

\begin{cor}\label{cor:if-Morley-then-total-ind-bu-Morley}
  If $(b_i)_{i<\omega}$ is a (non-dividing) Morley sequence over $A$, then it is a total $\indbu$-Morley sequence over $A$. \hfill $\qed$
\end{cor}

 In simple theories, we get the converse.

  \begin{prop}\label{prop:simple-char-ind-bu-Morley-seq}
    Let $T$ be a simple theory. For any $A$ and $A$-indiscernible sequence $I$, the following are equivalent.
    \begin{enumerate}[(1)]
    \item $I$ is an $\indf$-Morley sequence over $A$.
    \item For any $J$ and $K$ with $J+K \equiv^{\EM}_A I$, $J \indb_A K$ (i.e., $\bdd^{\heq}(AJ)\cap \bdd^{\heq}(AK) = \bdd^{\heq}(A)$).
    \item $I$ is a total $\indbu$-Morley sequence over $A$.
    \end{enumerate}
  \end{prop}
  \begin{proof}     (1)$\To$(3) is \cref{cor:if-Morley-then-total-ind-bu-Morley}. (3)$\To$(2) is obvious.

    (2)$\To$(1) follows by the same reasoning as in \cite[Sec.~3.1]{Adler2005ExplanationOI}. Specifically, in simple theories, $b \indf_AC$ (with $A \subseteq C$) holds if and only if $\cb(\stp(b/C)) \subseteq \bdd^{\heq}(A)$. This implies that $\indf$ satisfies intersection over hyperimaginaries (i.e., if $A_0 \subseteq C$ $A_1\subseteq C$, $b \indf_{A_0} C$, and $b \indf_{A_1} C$, then $b \indf_{\bdd^{\heq}(A_0)\cap \bdd^{\heq}(A_1)} C$). By the same argument as in \cite[Lem.~3.4]{Adler2005ExplanationOI}, we have that $b \indf_A C$ if and only if
    \begin{itemize}
    \item[$(\ast)$]      there is an $AC$-indiscernible sequence $(b_i)_{i <\omega}$ with $b_0 = b$ such that for any $J$ and $K$ with $J+K \equiv^{\EM}_A b_{<\omega}$, $J \indb_A K$.
    \end{itemize}
    Assume that (2) holds. Fix $(b_i)_{i < \omega+\omega}\equiv^{\EM}_A I$. $(b_i)_{\omega\leq i < \omega+\omega}$ is $Ab_{<\omega}$-indiscernible and satisfies $(\ast)$. Therefore $b_\omega \indf_A b_{<\omega}$, and we have that $b_{<\omega+\omega}$, and therefore $I$, is an $\indf$-Morley sequence over $A$.
      \end{proof}

  On the other hand, there are easy examples in NIP theories (such as $\mathsf{DLO}$) of total $\indbu$-Morley sequences that are not strict Morley sequences; any sequence generated by an $A$-invariant type that is not strictly $A$-invariant will be a total $\indbu$-Morley sequence over $A$ but not a strict Morley sequence over $A$. $\mathsf{DLO}$ can also be used to show that not every $\indt$-Morley sequence in a rosy theory is a total $\indbu$-Morley sequence (e.g., \cite[Ex.~3.13]{Adler2005ExplanationOI}).

In NSOP$_1$ theories, we get that tree Morley sequences are total $\indbu$-Morley sequences.

  \begin{prop}\label{prop:total-ind-bu-Morley-seq-in-NSOP1}
    Let $T$ be an NSOP$_1$ theory, and let $M \models T$. If $I$ is a tree Morley sequence over $M$, then it is a total $\indbu$-Morley sequence over $M$.
  \end{prop}
  \begin{proof}
    Since we are working over a model, Lascar strong types are types. Let $J$ be a sequence realizing the same EM-type as $I$ over $M$. Find $K \equiv_M I$ such that $K \indK_M IJ$. Let $I'$, $J'$, and $K'$ have the same order type such that $I+I'$, $J+J'$, and $K+K'$ are all $M$-indiscernible. Since these are tree Morley sequences, we have that $I\indK_M I'$, $J\indK_M J'$, and $K\indK_M K'$. Therefore, by the independence theorem for NSOP$_1$ theories, we can find $I''$ and $J''$ such that $I+I''$, $K+I''$, $K+J''$, and $J+J''$ are all $M$-indiscernible, so $I \approx_M J$.   

    Since we can do this for any such $J$, we have that $I$ is a total $\indbu$-Morley sequence.
  \end{proof}

  The converse is unclear. The argument in the context of simple theories relies on the existence of canonical bases for types.

  \begin{quest}
    If $T$ is NSOP$_1$, is every total $\indbu$-Morley sequence over $M\models T$ a tree Morley sequence over $M$?
  \end{quest}

\bibliographystyle{plain}
\bibliography{../ref}

\end{document}